\documentclass[11pt]{amsart}
\usepackage{url}
\usepackage{amssymb,amscd}
\usepackage[all]{xy}

 \usepackage{mathrsfs,amsfonts,amsmath,amsthm,amssymb}
\usepackage{graphics}
\usepackage[all]{xy}
\xyoption{curve}
\xyoption{import}
\xyoption{arc}
\xyoption{ps}
\usepackage{amsmath}
\usepackage{graphicx}

  \usepackage[a4paper, margin=3cm]{geometry}

 \usepackage{times}

\begin{document}

\numberwithin{equation}{section}

{\theoremstyle{definition}\newtheorem{definition}{Definition}[section]
\newtheorem{notation}[definition]{Notation}
\newtheorem{remnot}[definition]{Remarks and notation}
\newtheorem{terminology}[definition]{Terminology}
\newtheorem{remark}[definition]{Remark}
\newtheorem{remarks}[definition]{Remarks}
\newtheorem{example}[definition]{Example}
\newtheorem{examples}[definition]{Examples}
\newtheorem{deflem}[definition]{Definition-Lemma}
\newtheorem{proposition}[definition]{Proposition}
\newtheorem{lemma}[definition]{Lemma}
\newtheorem{theorem}[definition]{Theorem}
\newtheorem{corollary}[definition]{Corollary}

\newcommand{\ci}{C^{\infty}}
\newcommand{\A}{\mathscr{A}}
\newcommand{\Cat}{\mathscr{C}}
\newcommand{\Dnc}{\mathscr{D}}
\newcommand{\E}{\mathscr{E}}
\newcommand{\F}{\mathscr{F}}
\newcommand{\gr}{\mathscr{G}}
\newcommand{\go}{\mathscr{G} ^{(0)}}
\newcommand{\hr}{\mathscr{H}}

\newcommand{\cH}{\mathcal{H}}
\newcommand{\cU}{\mathcal{U}}

   \newcommand{\sta}{\stackrel}
   \newcommand{\ba}{\begin{eqnarray}}
   \newcommand{\na}{\end{eqnarray}}
   \newcommand{\ban}{\begin{eqnarray*}}
   \newcommand{\nan}{\end{eqnarray*}}

\newcommand{\ho}{\mathscr{H} ^{(0)}}
\newcommand{\rgr}{\mathscr{R}}
\newcommand{\rgo}{\mathscr{R} ^{(0)}}
\newcommand{\lr}{\mathscr{L}}
\newcommand{\lo}{\mathscr{} ^{(0)}}
\newcommand{\gd}{\mathscr{G}^{\mathbb{R}^2}}
\newcommand{\gt}{\mathscr{G} ^{T}}
\newcommand{\I}{\mathscr{I}}
\newcommand{\Nb}{\mathscr{N}}
\newcommand{\Kom}{\mathscr{K}}
\newcommand{\ops}{\mathscr{O}}
\newcommand{\Pb}{\mathscr{P}}
\newcommand{\sw}{\mathscr{S}}
\newcommand{\Uo}{\mathscr{U}}
\newcommand{\Vo}{\mathscr{V}}
\newcommand{\Wo}{\mathscr{W}}
\newcommand{\R}{\mathbb{R}}
\newcommand{\Nat}{\mathbb{N}}
\newcommand{\C}{\mathbb{C}}
\newcommand{\src}{\mathscr{S}_{c}}
\newcommand{\cg}{C_{c}^{\infty}(\gr)}
\newcommand{\cgo}{C_{c}^{\infty}(\go)}
\newcommand{\ct}{C_{c}^{\infty}(\gr^T)}
\newcommand{\ca}{C_{c}^{\infty}(A\gr)}
\newcommand{\Un}{{U}^{(n)}}
\newcommand{\Du}{D_{\mathscr{U}}}

\def\to{\longrightarrow}

\def\L{\mathop{\wedge}}

\def\gpd{\,\lower1pt\hbox{$\longrightarrow$}\hskip-.24in\raise2pt
             \hbox{$\longrightarrow$}\,}

\title[PDO calculus, twisted gerbes and twisted index theory]
{Pseudodifferential calculus, twisted gerbes and twisted index theory for Lie groupoids}

 \author[Paulo Carrillo Rouse  ]{Paulo Carrillo Rouse }
  \address[Paulo Carrillo Rouse  ]
  {Institut de Math\'ematiques de Toulouse\\ 118, route de Narbonne\\
  31062 Toulouse, France}
  \email{carrillo@math.univ-toulouse.fr}

 % \author[Bai-Ling Wang]{\sc Bai-Ling Wang}
  %\address[Bai-Ling Wang]
  %{Department of Mathematics\\
  %Mathematical Sciences Institute\\
  %Australian National University\\
  %Canberra ACT 0200 \\
  %Australia}
  %\email{wangb@maths.anu.edu.au}

   \thanks{
 }

\bigskip
\everymath={\displaystyle}

\begin{abstract}
\noindent The goal of this paper is to construct a calculus whose higher indices are naturally elements in the twisted K-theory groups for Lie groupoids. 
Given a Lie groupoid $G$ and a $PU(H)$-valued groupoid cocycle, we construct an algebra of projective pseudodifferential operators. The subalgebra of regularizing operators identifies with the naturally associated smooth convolution algebra of the associated twisted gerbe. We develop the associated symbolic calculus, symbol short exact sequences and existence of parametrices. In particular the algebra of projective operators appears as a quantization of the twisted symbol algebra. As the (untwisted) Lie groupoid case that it encompasses, the negative order operators extend to the twisted $C^*$-algebra and the zero order operators act as bounded multipliers on it. 
We obtain an analytic index morphism in twisted K-theory associated in a classic way by the corresponding pseudodifferential extension. We prove that this index morphism only depends on the isomorphism class of the cocycle, {\it i.e.,} on the twisting as the associated class in $H^1(G;PU(H))$. We also show that this twisted analytic index morphism is compatible with the index we constructed in a previous work, in collaboration with Bai-Ling Wang, by means of the Connes tangent groupoid, obtaining as a consequence the analytic interpretation, in terms of projective pseudodifferential operators and ellipticity, of the twisted longitudinal Connes-Skandalis index theorem.
Our construction encompasses and unifies several previous cases treated in the literature, we discuss in the final section some examples of classes of operators unified by our setting.
%(1) The projective longitudinal families of Dirac operators which generalize the case of projective families of Dirac operators described by Mathai, Melrose and Singer, and (2) The projective symbols of Fractional Indices' projective pseudodifferential operators. 

%for instance it generalizes the projective families of pseudodifferential operators introduced by Mathai, Melrose and Singer.  We end by giving three examples of classes of operators unified by our setting, the projective longitudinal families of Dirac operators associated to any Lie algebroid, the Dirac operator associated to a $\Gamma$-covering twisted by a projective representation and the projective symbols of Fractional Indices' projective pseudodifferential operators. 

%\noindent
\end{abstract}

\maketitle

\tableofcontents

%\newpage
\section{Introduction}

%%%%%%%%%%Some History on the subject%%%%%%%%%%%%%%%%%%%%%%%%%%%%%%%%%%%%%%%%%%%%%%%%

In a series of papers, \cite{MMSI,MMSII}, Mathai, Melrose and Singer extended the classic Atiyah-Singer index theory to the realm of projective operators for the case of families\footnote{The case of fractional indices' operators is slighty different, we discuss it with detail in the last section.}. Even if some examples of projective operators appeared before in the literature it was in their papers that the notions of projective vector bundles and projective pseudodifferential operators with respect to some Azumaya bundle were formalized at least in the cases were this Azumaya bundle is obtained from a torsion class. Our motivation for study projective operators comes from twisted index theory and twisted K-theory. In the same way that "classic" index theory is indissociable from "classic" topological K-theory, Mathai, Melrose and Singer's fundamental tools and techniques relied on twisted K-theory. The subject of Twisted K-theory has been growing very fast in the last years due to its deep relations with several domains in mathematics and mathematical physics.

Now, twisted K-theory can be defined for higher structures such as Lie groupoids, which cover very diverse and different geometric situations: Groups (Lie and discrete), families, orbifolds, foliations, coverings, singular manifolds for mention some of them. In fact in \cite{TXL} Tu, Xu and Laurent-Gengoux developped all the properties of the twisted K-theory for differentiable stacks mainly in terms of the $K-$theory of the associated twisted $C^*$-algebra but also doing the link with more topological models whenever is possible (for proper groupoids and for torsion twistings). 

In a series of papers, in collaboration with Bai-Ling Wang \cite{CaWangCRAS,CaWangAdv,CaWangAENS} we have been studying index theory for twisted Lie groupoids, mainly following the approach inspired by the Connes' tangent groupoid.

In this paper we develop the pseudodifferential calculus that corresponds to the twisted K-theory for Lie groupoids, showing how to extend the classic notions to the twisted world and showing in particular how twisted K-theory is the natural receptacle of the higher indices of elliptic projective pseudodifferential operators. A proper construction and main properties of this calculus were missing, even in the case of famililes our results complete and explain some of the classic properties. We explain next in more detail the contents of the present article.

Let $\gr\rightrightarrows M$ be a Lie groupoid with compact smooth base $M$. Assume there is a twisting $\alpha$ on $\gr$, that is, the isomorphism class of a generalized morphism $\alpha:\gr--->PU(H)$ or equivalently the isomorphism class of a $PU(H)$-principal bundle over $M$ with a compatible $\gr$-action. The twisting $\alpha$ can be hence represented by a groupoid morphism (a cocycle)
$$\alpha_\Omega:\gr_\Omega\to PU(H)$$
associated to an good\footnote{In practice we will always assume in this paper that the open covers of the base are locally finite and we could even assume the open subsets $\Omega_i$ are manifold charts.} open cover $\Omega$ of $M$ and where $\gr_\Omega$ stands for the naturally associated Cech groupoid. In practice is often common to start with a cocycle to get afterwards the associated class.

To better explain the construction of the algebra of Pseudodifferential operators is worthy to recall how the twisted convolution algebra $C_c^\infty(\gr,\alpha_\Omega)$ is defined and how the product is given. Denote by $\gr_i^j:=s^{-1}(\Omega_j)\cap t^{-1}(\Omega_i)$ where $s,t$ are the source and target maps for the groupoid $\gr$ and $\Omega:=\{\Omega_i\}_i$ is the open cover of the base. The morphism $\alpha_\Omega$ gives in a canonical way a line bundle $L_{ij}$ over each $\gr_i^j$, the union $L:=\sqcup_{(i,j)}L_{ij}\to \gr_\Omega$ is a Fell bundle in the sense that we have  
isomorphisms
$$L_{ij}^g\times L_{jk}^h\stackrel{\bullet_F}{\longrightarrow} L_{ik}^{hg}$$
whenever $\gr_i^j \,_{t_j}\times_{s_j}\gr_j^k$ not empty and $(g,h)\in \gr_i^j \,_{t_j}\times_{s_j}\gr_j^k$. The algebra $C_c^\infty(\gr,\alpha_\Omega)$ is by definition the compactly supported $C^\infty$-sections $$C_c^\infty(\gr_\Omega,L\otimes \Omega^{\frac{1}{2}})$$
where $\Omega^{\frac{1}{2}}|_{\gr_i^j}:=\Omega^{\frac{1}{2}}(Ker\,ds_j\oplus Ker\, dt_i),$ is a bundle of half densities, and with product 
$$C_c^\infty(\gr_i^j,L_{ij}\otimes \Omega^{\frac{1}{2}})\times C_c^\infty(\gr_j^k,L_{jk}\otimes \Omega^{\frac{1}{2}}) \to C_c^\infty(\gr_i^k,L_{ik}\otimes \Omega^{\frac{1}{2}})$$ defined by
$$(f*g)(\gamma):=\int_{(\gamma_1,\gamma_2)\in \gr_\Omega:\gamma_1\cdot \gamma_2=\gamma} f(\gamma_1)\bullet_Fg(\gamma_2)$$
for $f\in C_c^\infty(\gr_i^j,L_{ij}\otimes \Omega^{\frac{1}{2}})$, $g\in C_c^\infty(\gr_j^k,L_{jk}\otimes \Omega^{\frac{1}{2}})$, where, as usual, the integral is the integral of a canonically associated $1$-density obtained by product of two half densities together with the Fell product $\bullet_F$.

We will define an algebra of pseudodifferential distributions that extends the product above in a natural way. For each $(i,j)\in I^2$ as above one can consider the couple of manifold-submanifold $(\gr_i^j,\Omega_{ij})$ where $\Omega_{ij}$ stands for the intersection (it might be empty of course). We can consider the spaces of compactly supported generalized sections with pseudodifferential singularities on $\Omega_{ij}$ of order $m$
$$P_c^m(\gr_i^j,\Omega_{ij};L_{ij}\otimes \Omega^{\frac{1}{2}}),$$
see section \ref{ASsection} below for more details. It is essentially a space of conormal distributions associated to the couple $(\gr_i^j,\Omega_{ij})$. We let as usual 
$$P_c^\infty(\gr,\alpha_\Omega)=\bigcup_{m\in \mathbb{Z}}P_c^m(\gr,\alpha_\Omega)$$
and  
$$P_c^{-\infty}(\gr,\alpha_\Omega)=\bigcap_{m\in \mathbb{Z}}P_c^m(\gr,\alpha_\Omega).$$

The product
\begin{equation}
\xymatrix{
P_c^m(\gr_i^j,\Omega_{ij};\Omega^{\frac{1}{2}}\otimes L_{ij})\times P_c^n(\gr_j^k,\Omega_{jk};\Omega^{\frac{1}{2}}\otimes L_{jk})\ar[r]&P_c^{m+n}(\gr_i^k,\Omega_{ik};\Omega^{\frac{1}{2}}\otimes L_{ik})
}
\end{equation}
defined in section \ref{pdogrpds} is given as follows: Let $K_{ij}\in P_c^m(\gr_i^j,\Omega_{ij};\Omega^{\frac{1}{2}}\otimes L_{ij})$ and $K_{jk}\in P_c^n(\gr_j^k,\Omega_{jk};\Omega^{\frac{1}{2}}\otimes L_{jk})$. By proposition \ref{ASproduct} (a) the product
$$p_1^*(K_{ij})\cdot p_2^*(K_{jk})$$ makes sense as a distribution acting on  
$$C_c^\infty(\gr_i^j\,_{t_j}\times_{s_j}\gr_j^k,\Omega^1(\gr_i^j\,_{t_j}\times_{s_j}\gr_j^k)\otimes (p_1^*(\Omega^{\frac{1}{2}}\otimes L_{ij})\otimes p_2^*(\Omega^{\frac{1}{2}}\otimes L_{jk}))^*),$$
where $p_1,p_2$ stand for the canonical projections. By lemma \ref{LMVlemma}, the last space is isomorphic to 
$$C_c^\infty(\gr_i^j\,_{t_j}\times_{s_j}\gr_j^k,\Omega^1(\gr_i^j\,_{t_j}\times_{s_j}\gr_j^k)\otimes(\Omega^1Ker\,dm\otimes m^*(\Omega^{\frac{1}{2}}\otimes L_{ik}))^*),$$ 
where $m$ denotes the groupoid multiplication map.
We can now apply the external product construction \ref{ASproduct} (b) to define 
\begin{equation}\label{twistdistproductintro}
K_{ij}*K_{jk}:=m!(p_1^*(K_{ij})\cdot p_2^*(K_{jk}))\in P_c^{m+n}(\gr_i^k,\Omega_{ik};\Omega^{\frac{1}{2}}\otimes L_{ik})
\end{equation}

The main result of this paper is the following (theorem \ref{pdoalg})
%By definition this last space identifies with $C_c^\infty(\gr,\alpha_\Omega)$, our product is hence
\begin{theorem}\label{pdoalgintro}
With the product (\ref{twistdistproductintro}) and the involution (\ref{twistdistinvolution}) described below, the union $$P_c^\infty(\gr,\alpha_\Omega)$$
forms a filtered $*-$algebra with bilateral ideal 
$$P_c^{-\infty}(\gr,\alpha_\Omega)\cong C_c^\infty(\gr,\alpha_\Omega).$$
\end{theorem}

The above result requires first to understand several operations on Androulidakis-Skandalis' generalized distributions, pullback, pushforward, external product. We study these operations and their properties (functoriality, compatibility between them, etc.) in section \ref{ASsection}. Then in section \ref{pdogrpds} we define the product and prove the theorem (the associativity of the product requires special attention). 

We show afterwards, as in the untwisted case, that the algebra above can be explicitly realized as an algebra of multipliers on $C_c^\infty(\gr,\alpha_\Omega)$, corollary \ref{smoothmultipliers}, that we denote in a more familiar way by 
$$\Psi_c^\infty(\gr,\alpha_\Omega)$$
with bilateral ideal 
$$\Psi_c^{-\infty}(\gr,\alpha_\Omega).$$

%The main result of this paper (theorem \ref{pdoalg}) is to construct for every $(\gr,\alpha_\Omega)$ as above a $\mathbb{Z}$-graded $*-$algebra of pseudodifferential operators 

For example, if $\alpha_\Omega$ is trivial, then the algebra above gives precisely the algebra of $\gr$- pseudodifferential operators independently introduced in \cite{MP} and in \cite{NWX}, and which already covers various very interesting geometric situations. In fact for an expert knowing about this groupoid calculus is tempting to adapt the invariant families definition ($\gr$-operators) from the previous references to include the twisting, however this becomes quickly complicated when using twisted gerbes over groupoids, and even more for proving all the desired properties (symbolic calculus, extensions, parametrices, etc.), instead we decided to follow a similar approach to the one in \cite{ASpdo} where Androulidakis and Skandalis construct the pseudodifferential calculus for singular foliations, we define then our operators by means of generalized distributions with pseudodifferential singularities and we used the Fell line bundle naturally associated to the twisting gerbe to properly define the product. Our product is then a natural extension of the twisted convolution product in $C_c^\infty(\gr,\alpha_\Omega)$.
We prove afterwards that our operators define indeed families of operators invariant under a twisted action as expected, section \ref{twistedfamiliessection}. In the case the twisting $\alpha$ is torsion one can extend the operator algebra above to include coefficients on twisted vector bundles.

%We expect to convince the reader that it is very natural in the twisted case to take this approach. 

As far as we know, if $\alpha$ is not necessarily trivial, the more general case covered before the present work (not in terms of groupoids) is for $\gr$ the groupoid associated associated to a fibration $\phi:M\to X$  and $\alpha$ torsion and it was treated in the seminal work by Mathai, Melrose and Singer (\cite{MMSI}) where they started the development of index theory in this context, the case where $\alpha$ is supposed to be torsion once pullbacked to $M$ is studied in \cite{BG} where Benameur and Gorokhovsky also prove a local index theorem, we were very much inspired by this last paper. After the release of this paper (february 2016 on the arxiv), Benameur, Gorokhovsky and Leichtnam released (july 2016 on the arxiv) the article \cite{BGL} in which they describe part of the calculus for foliations, this is another example that fits in our setting. Of course in the last ref.cit. paper the authors go further and prove a higher index theorem. Other particular projective operators have been treated in the litterature, we will mention some of them later on the paper.

We continue with the description of the contents of this paper. We develop the associated symbolic calculus, symbol short exact sequences and existence of parametrices. In particular the algebra of projective operators appears as a quantization of the twisted symbol algebra that we properly introduce in section \ref{twistsymbsec}. 

As the (untwisted) Lie groupoid case that it encompasses, the negative order operators extend to the twisted $C^*$-algebra and the zero order operators act as bounded multipliers on it.
 
We obtain an analytic index morphism in twisted K-theory associated in a classic way by the corresponding pseudodifferential extension, definition \ref{defanaindex},
\begin{equation}\label{aindexintro}
\xymatrix{
K^1(S^*\gr,\pi^*\alpha^0)\ar[rr]^-{Ind^{a}_{(\gr,\alpha)}}&&K^0(\gr,\alpha)
}
\end{equation} 
where $K^0(\gr,\alpha):=K_0(C^*(\gr,\alpha))$.
We prove that this index morphism only depends on the isomorphism class of the cocycle, {\it i.e.,} on the twisting as the associated class in $H^1(G;PU(H))$ (proposition \ref{AIndexinvariance}), of course, as the untwisted case, it is not a Morita invariant of the correspondent twisted differentiable stack. The one that is Morita invariant is the Baum-Connes assembly map constructed in \cite{CaWangAENS}.
 
We also show that the analytic index morphism above factors in a canonical way by the index we constructed in our previous work \cite{CaWangAdv} by means of the Connes tangent groupoid, theorem \ref{DefindexvsAnaindex}, obtaining as a consequence the analytic interpretation, in terms of pseudodifferential operators and ellipticity, of the twisted longitudinal Connes-Skandalis index theorem, theorem \ref{twistCSthm}.

In the final section we discuss two examples of classes of operators unified by our setting:
\begin{enumerate} 
\item The projective longitudinal families of Dirac operators associated to any torsion twisted groupoid, for which we describe two explicit subexamples, one case for foliations that includes the case of fibrations treated in \cite{MMSI} by Mathai, Melrose and Singer or in \cite{BG} by Benameur and Gorokhovsky and the case treated by Benameur, Gorokhovsky and Leichtnam in \cite{BGL}; and one case for a $\Gamma$-covering twisted by a projective representation. 
\item The projective symbols of Fractional Indices' projective pseudodifferential operators. In this last paragraph we discuss in detail the case of operators introduced in \cite{MMSfrac}, we explain that these are not the kind of operators treated in this article but that the their total symbols are. 
\end{enumerate}

{\bf Acknowledgements:} This project born during the tea breaks discussions with Bai-Ling Wang at the Max Planck Institut fur Mathematics at Bonn in 2008. The article was continued and completed during a stay at the MPIM of the author in 2015, I am very grateful to this Institution for the excellent working conditions I enjoyed during the realization and preparation of this work. I want to thank Bai-Ling Wang, who was part of this project in an earlier stage, for introduce me to twisted world and for all the past and future twisted (or untwisted) projects those tea breaks discussions have generated.

\section{Preliminaries on twistings on groupoids and twisted algebras}

In this section, we review the notion of twistings on Lie groupoids and their $C^*$-algebras and 
discuss some examples which appear in this paper.
Let us recall what a groupoid is:

\begin{definition}
A $\it{groupoid}$ consists of the following data:
two sets $\gr$ and $\go$, and maps
\begin{itemize}
\item[(1)]  $s,r:\gr \rightarrow \go$ 
called the source map and target map respectively,
\item[(2)]  $m:\gr^{(2)}\rightarrow \gr$ called the product map 
(where $\gr^{(2)}=\{ (\gamma,\eta)\in \gr \times \gr : s(\gamma)=r(\eta)\}$),
\end{itemize}
together with  two additional  maps, $u:\go \rightarrow \gr$ (the unit map) and 
$i:\gr \rightarrow \gr$ (the inverse map),
such that, if we denote $m(\gamma,\eta)=\gamma \cdot \eta$, $u(x)=x$ and 
$i(\gamma)=\gamma^{-1}$, we have 
\begin{itemize}
\item[(i)] $r(\gamma \cdot \eta) =r(\gamma)$ and $s(\gamma \cdot \eta) =s(\eta)$.
\item[(ii)] $\gamma \cdot (\eta \cdot \delta)=(\gamma \cdot \eta )\cdot \delta$, 
$\forall \gamma,\eta,\delta \in \gr$ whenever this makes sense.
\item[(iii)] $\gamma \cdot u(x) = \gamma$ and $u(x)\cdot \eta =\eta$, $\forall
  \gamma,\eta \in \gr$ with $s(\gamma)=x$ and $r(\eta)=x$.
\item[(iv)] $\gamma \cdot \gamma^{-1} =u(r(\gamma))$ and 
$\gamma^{-1} \cdot \gamma =u(s(\gamma))$, $\forall \gamma \in \gr$.
\end{itemize}
For simplicity, we denote a groupoid by $\gr \rightrightarrows \go $. A strict morphism $f$ from
a  groupoid   $\hr \rightrightarrows \ho $  to a groupoid   $\gr \rightrightarrows \go $ is  given
by  maps 
\[
\xymatrix{
\hr \ar@<.5ex>[d]\ar@<-.5ex>[d] \ar[r]^f& \gr \ar@<.5ex>[d]\ar@<-.5ex>[d]\\
\ho\ar[r]_{f_0}&\go
}
\]
which preserve the groupoid structure, i.e.,  $f$ commutes with the source, target, unit, inverse  maps, and respects the groupoid product  in the sense that $f(h_1\cdot h_2) = f (h_1) \cdot f(h_2)$ for any $(h_1, h_2) \in \hr^{(2)}$.

\end{definition}

In  this paper we will only deal with Lie groupoids, that is, 
a groupoid in which $\gr$ and $\go$ are smooth manifolds, and $s,r,m,u$ are smooth maps (with s and r submersions, see \cite{Mac,Pat}). 
 
Lie groupoids form a category with  strict  morphisms of groupoids. It is now a well-established fact  in Lie groupoid's theory that the right category to consider is the one in which Morita equivalences correspond precisely to isomorphisms.  We review some basic definitions and properties of generalized morphisms between Lie groupoids, see \cite{TXL} section 2.1, or 
\cite{HS,Mr,MM} for more detailed discussions.

\begin{definition}[Generalized morphisms]\label{HSmorphism}   
Let $\gr \rightrightarrows \go$ and  
$\hr \rightrightarrows \ho$ be two Lie groupoids.  A generalized groupoid morphism, also called a Hilsum-Skandalis morphism, from $\hr$ to $\gr$ is given by  principal $\gr$-bundle over $\hr$, that 
is, a right  principal $\gr$-bundle over $\ho$
which is also a left $\hr$-bundle over $\go$ such that the   the right $\gr$-action and the left 
$\hr$-action commute,  formally denoted by
\[
f:  \xymatrix{\hr \ar@{-->}[r] &  \gr}
\]
or by  
\[
\xymatrix{
\hr \ar@<.5ex>[d]\ar@<-.5ex>[d]&P_f \ar@{->>}[ld] \ar[rd]&\gr \ar@<.5ex>[d]\ar@<-.5ex>[d]\\
\ho&&\go.
}
\]
if   we want to emphasize  the bi-bundle $P_f$ involved. 
\end{definition}

Notice that a generalized morphism (or Hilsum-Skandalis morphism),   
$f:  \xymatrix{\hr \ar@{-->}[r] &  \gr}$, is given by one of the three equivalent data:
\begin{enumerate}
\item A  locally trivial  right  principal $\gr$-bundle $P_f$  over  $\hr$  as Definition \ref{HSmorphism}. 
%\[ \xymatrix{ \hr \ar@<.5ex>[d]\ar@<-.5ex>[d]&P_f \ar@{->>}[ld] \ar[rd]&\gr \ar@<.5ex>[d]\ar@<-.5ex>[d]\\ \ho&&\go.} \]
\item A 1-cocycle $f=\{(\Omega_i,f_{ij})\}_{i\in I}$ on $\hr$ with values in $\gr$. Here a  $\gr$-valued 1-cocycle on  $\hr$ with respect to  an indexed open covering $\{\Omega_i\}_{i\in I}$ of $\ho$ is  a collection of smooth maps 
 $$f_{ij}:\hr_{\Omega_j}^{\Omega_i} \to\gr,$$
 satisfying the following cocycle condition:
$\forall \gamma \in \hr_{ij}$ and $\forall \gamma'\in \hr_{jk}$ with $s(\gamma)=r(\gamma')$, we have
\begin{center}
$f_{ij}(\gamma)^{-1}=f_{ji}(\gamma^{-1})$ and $f_{ij}(\gamma)\cdot f_{jk}(\gamma')=f_{ik}(\gamma\cdot \gamma').$
\end{center}
We will denote this data by $f=\{(\Omega_i,f_{ij})\}_{i\in I}$. 

\item A  strict morphism of groupoids 
 
\begin{equation}\nonumber
\xymatrix{
\hr_{\Omega}=\bigsqcup_{i,j}\hr_{\Omega_j}^{\Omega_i} \ar@<.5ex>[d]\ar@<-.5ex>[d]\ar[rr]^-f &&\gr\ar@<.5ex>[d]\ar@<-.5ex>[d]\\
 \bigsqcup_{i}\Omega_{i}\ar[rr]&&\go.
}
\end{equation}
for an open cover  $\Omega= \{\Omega_i\}$ of $\ho$.
 \end{enumerate}

 Associated to  a $\gr$-valued 1-cocycle on  $\hr$, there is a canonical defined  principal $\gr$-bundle over $\hr$.  In fact, any principal $\gr$-bundle over $\hr$ is locally trivial (Cf. \cite{MM}).  

\vspace{2mm}
  
\begin{example}  \begin{enumerate}
\item  (Strict morphisms)
Consider a (strict) morphism of groupoids
\[
\xymatrix{
\hr \ar@<.5ex>[d]\ar@<-.5ex>[d] \ar[r]^f& \gr \ar@<.5ex>[d]\ar@<-.5ex>[d]\\
\ho\ar[r]_{f_0}&\go
}
\]
Using the equivalent definitions 2. or 3. above, it is obviously a generalized morphism by taking $\Omega=\{\ho\}$.
In terms of  the language of principal bundles,  the bi-bundle  is simply given  by $$P_f:=\ho\times_{f_0,t}\gr,$$
with projections $t_f:P_f\to \ho$, projection in the first factor, and 
$s_f:P_f\to \go$, projection using the source map of $\gr$. The actions are the obvious ones, that is,
on the left, $h\cdot (a,g):=(t(h),f(h)\circ g)$ whenever $s(h)=a$ and, on the right, $(a,g)\cdot g':=(a,g\circ g')$ whenever $s(g)=t(g')$.
 \item (Classic principal bundles)
Let $X$ be a manifold and $G$ be a Lie group. By definition a generalized morphism between the unit groupoid $X\rightrightarrows X$ (that is a manifold seen as a Lie groupoid all structural maps are the identity) and the Lie group $G\rightrightarrows \{e\}$ seen as a Lie groupoid is given by a $G$-principal bundle over $X$. 

\end{enumerate}
\end{example}

As the name suggests,  generalized morphism  generalizes the notion of strict morphisms and can be composed. Indeed, if $P$ and $P'$ are generalized morphisms from $\hr$ to $\gr$ and from $\gr$ to $\lr$ respectively, then 
$$P\times_{\gr}P':=P\times_{\go}P'/(p,p')\sim (p\cdot \gamma, \gamma^{-1}\cdot p')$$
is a generalized morphism from $\hr$ to $\lr$.  Consider the category $Grpd_{HS}$ with objects Lie groupoids and morphisms given by isomorphism classes of generalized morphisms. There is a functor
\begin{equation}\label{grpdhs}
Grpd \longrightarrow Grpd_{HS}
\end{equation}
where $Grpd$ is the strict category of groupoids. 
%Then the composition is associative in $Grpd_{HS}$. 
  
 \subsection{Twistings on  Lie groupoids}
 
 In this paper,  we are going to consider $PU(H)$-twistings on Lie groupoids 
where $H$ is an infinite dimensional, complex and separable
Hilbert space, and $PU(H)$ is the projective unitary group $PU(H)$  with the topology induced by the
norm topology on the unitary group  $U(H)$. 

\begin{definition}\label{twistedgroupoid}
A  twisting $\alpha$  on a   Lie  groupoid $\gr \rightrightarrows \go$  is given by the isomorphism class of a  generalized morphism 
\[ \xymatrix{
\alpha: \gr \ar@{-->}[r]  & PU(H).}
\]
Here $PU(H)$ is viewed  as a Lie groupoid with the unit space $\{e\}$. Two twistings 
$\alpha$ and $\alpha'$ are called equivalent if they are  equivalent as generalized morphisms.
\end{definition}
 
 So a twisting on a Lie groupoid $\gr$ can be represented by a locally trivial  right  principal $PU(H)$-bundle $P_{\alpha}$ over $\gr$
%\begin{equation}\label{gpalpha} \xymatrix{ \gr \ar@<.5ex>[d]\ar@<-.5ex>[d]&P_{\alpha} \ar[ld] \ar[rd]&PU(H) \ar@<.5ex>[d]\ar@<-.5ex>[d]\\\go&&\{e\}. } \end{equation}
hence,  represented by
 a $PU(H)$-valued 1-cocycle on $\gr$
\begin{equation}\label{galphaOmega}
\alpha_{ij}:   \gr_i^j \longrightarrow PU(H)
\end{equation}
for an open cover $\Omega= \{\Omega_i\}$ of $\go$ and where $\gr_i^j:=s^{-1}(\Omega_i)\cap t^{-1}(\Omega_j)$. That is,  a twisting  datum $\alpha$ on a Lie  groupoid $ \gr $ can be represented by a strict morphism of groupoids 
\begin{equation}\label{galpha}
\xymatrix{
\gr_{\Omega}=\bigsqcup_{i,j}\gr_i^j \ar@<.5ex>[d]\ar@<-.5ex>[d]\ar[rr]^-{\alpha_\Omega} &&PU(H) \ar@<.5ex>[d]\ar@<-.5ex>[d]\\
 \bigsqcup_{i}\Omega_{i}\ar[rr]&&\{e\}.
}
\end{equation}
for an open cover  $\Omega= \{\Omega_i\}$ of $\go$.

 \begin{remark}
 The definition of generalized morphisms given in the last subsection was for two Lie groupoids. The group $PU(H)$ it is not precisely a Lie group but it makes perfectly sense to speak of generalized morphisms from Lie groupoids to this infinite dimensional   groupoid following exactly the same definition,  see (\ref{galphaOmega}) and (\ref{galpha}).
 \end{remark}
 
 \begin{remark}
 In practice one restrict, without lost of generality, to the use of good locally finite open covers over the base manifold of the groupoid, by good we mean that the open subsets $\Omega_i$ can be supposed to be manifold charts. We will always assume our covers are of this kind. 
 \end{remark}

\begin{example} \label{example} For a list of  various twistings on some   standard groupoids see example 1.8 in \cite{CaWangAdv}. Here we will only  a few  basic examples used in this paper.
  
 \begin{enumerate}
\item (Twisting on manifolds)  Let $X$ be a $\ci$-manifold. We can consider the  Lie groupoid 
 $X\rightrightarrows X$  where every morphism is the identity over $X$.  A twisting on $X$ is
given by a locally trivial principal $PU(H)$-bundle over $X$, or equivalently,  
 a twisting on $X$ is defined by a strict homomorphism 
\[
\xymatrix{
X_{\Omega}=\bigsqcup_{i,j} \Omega_{i, j} \ar@<.5ex>[d]\ar@<-.5ex>[d]\ar[rr]^-f &&PU(H) \ar@<.5ex>[d]\ar@<-.5ex>[d]\\
 \bigsqcup_{i}\Omega_{i}\ar[rr]&&\{e\}.
}
 \]
with respect to an open cover $\{\Omega_i\}$ of $X$, where $\Omega_{ij} =\Omega_i \cap\Omega_j$. 
Therefore, the restriction of a twisting $\alpha$ on a  Lie groupoid $\gr \rightrightarrows \go$
to its unit $\go$ defines a twisting  $\alpha_0$ on the manifold $\go$.

\item\label{obundle} (Orientation twisting) Let $X$ be a  manifold with an oriented real vector bundle $E$. The  bundle $E \to X$ defines
a natural generalized morphism 
\[
\xymatrix{
X\ar@{-->}[r] & SO(n).}
\]
Note that the fundamental spinor  representation of   $Spin^c(n)$ gives rise to a commutative
diagram of Lie group homomorphisms
\[
\xymatrix{
Spin^c(n) \ar[d]   \ar[r] & U(\mathbb{C}^{2^{[n/2]}}) \ar[d] \\
SO(n) \ar[r] & PU(\mathbb{C}^{2^{[n/2]}}).}
\]
With a choice of inclusion $\mathbb{C}^{2^{[n/2]}}$ into a Hilbert space $H$, we have a canonical
twisting, called the orientation twisting, denoted by
\begin{equation}\label{otwistingX}
\xymatrix{
\o_{E}:  X\ar@{-->}[r] & PU(H).}
\end{equation}
If now $\gr\rightrightarrows X$ is a Lie groupoid and $E$ is an oriented $\gr$-vector bundle over $X$, we have in the same way an orientation twisting
\begin{equation}\label{otwisting}
\xymatrix{
\o_{E}:  \gr\ar@{-->}[r] &SO(n)\ar[r]& PU(H)}
\end{equation}
in the case where $E$ admits an $\gr$-invariant metric. 
%In particular when $\gr$ acts properly on $P$ and on $E$, \cite{PPT} proposition 3.14 and \cite{HF} theorem 4.3.4. 

\item (Pull-back twisting) Given a twisting $\alpha$ on $ \gr$ and  for any generalized 
homomorphism $\phi: \hr \to \gr$, there is a pull-back twisting 
\[\xymatrix{
\phi^*\alpha:  \hr  \ar@{-->}[r]  & PU(H)}
\]
defined by the composition of $\phi$ and $\alpha$.  In particular, 
for a continuous map $\phi: X\to Y$, a twisting $\alpha$ on $Y$ gives a pull-back twisting 
$\phi^*\alpha$ on $X$. The principal $PU(H)$-bundle over $X$ defines by $\phi^*\alpha$ is
the pull-back of the  principal $PU(H)$-bundle on $Y$ associated to $\alpha$.

\item (Twisting on fiber product groupoid)  Let $N\stackrel{p}{\rightarrow} M$ be a submersion. We consider the fiber product $N\times_M N:=\{ (n,n')\in N\times N :p(n)=p(n') \}$,which is a manifold because $p$ is a submersion. We can then take the groupoid 
$$N\times_M N\rightrightarrows N$$ which is  a subgroupoid of the pair groupoid 
$N\times N \rightrightarrows N$.  Note that this groupoid is in fact Morita equivalent to the groupoid $ M \rightrightarrows M$.  A twisting on  
 $N\times_M N\rightrightarrows N$ is  given by
a pull-back twisting from a   twisting on  $M$.  

\item (Twisting on the space of leaves of a foliation)
Let $(M,F)$ be a regular foliation with holonomy groupoid $\gr_{M }$. A twisting on the space of leaves  is by definition a twisting on the holonomy groupoid $\gr_{M }$. We will often use the notation 
\[\xymatrix{
M/F \ar@{-->}[r] & PU(H)}
\]
for the corresponding  generalized morphism.

Notice that by definition a twisting on the spaces of leaves is a twisting on the base $M$ which admits a compatible action of the holonomy groupoid. It is however not enough to have a twisting on base which is leafwisely constant, see for instance remark 1.4 (c) in \cite{HS}.

\end{enumerate}
\end{example}

\subsection{Twisted groupoid's $C^*$-algebras}
Let $(\gr,\alpha)$ be a twisted groupoid. With respect to a covering  $\Omega = \{\Omega_i\}$ of $\go$, the twisting $\alpha$ is given by a strict morphism of groupoids 
$$ \alpha_\Omega: \gr_{\Omega}   \longrightarrow PU(H),$$
where $\gr_{\Omega}$ is the covering groupoid associated to $\Omega$. 
Consider the central extension of groups
$$S^1 \longrightarrow U(H) \longrightarrow PU(H),$$
 we can pull it back to get a $S^1$-central extension of Lie groupoid $R_{\alpha}$  over $\gr_{\Omega}$ 
\begin{equation}
\xymatrix{
  S^1\ar[d]\ar[r]& S^1\ar[d]\\
R_{\alpha}\ar[d]\ar[r]&U(H)\ar[d]\\
\gr_{\Omega}\ar[r]_-{\alpha}&PU(H)\\
}
\end{equation}
In particular, $R_{\alpha}\rightrightarrows \bigsqcup_i\Omega_i$ is a Lie groupoid and $R_{\alpha}\longrightarrow \gr_{\Omega}$ is a $S^1$-principal bundle.

We recall the definition of the convolution  algebra and the $C^*$-algebra of a twisted Lie groupoid $(\gr, \alpha_\Omega)$ \cite{Ren87,TXL}:

\begin{definition}
Let $R_{\alpha}$ be the $S^1$-central extension of groupoids associated to a twisting $\alpha$. The convolution algebra of $(\gr, \alpha_\Omega)$ is by definition the following sub-algebra of 
$C_{c}^{\infty}(R_{\alpha})$: 
\begin{equation}
C_{c}^{\infty}(\gr, \alpha_\Omega)=\{f\in C_{c}^{\infty}(R_{\alpha}): f(\tilde{\gamma} \cdot \lambda)=\lambda^{-1} f(\tilde{\gamma}), \forall \tilde{\gamma}\in R_{\alpha },\forall \lambda \in S^1\}.
\end{equation}
The maximal(reduced resp.) $C^*$-algebra of $(\gr, \alpha_\Omega)$, denoted by $C^*(\gr,  \alpha_\Omega)$ ($C^*_r(\gr, \alpha_\Omega)$ resp.),  is the completion of 
$C_{c}^{\infty}(\gr, \alpha_\Omega)$ in $C^*(R_{\alpha})$ ($C_r^*(R_{\alpha})$ resp.).
\end{definition}

Let $L_{\alpha}:=R_{\alpha}\times_{S^1}\mathbb{C}$ be the  complex line bundle  over $\gr_\Omega$ which can be considered as a Fell bundle using the groupoid structure of $R_{\alpha}$ over $\gr_{\Omega}$, indeed, denote by $L_{ij}$ the line bundle over $\gr_{ij}$, the fiber over a given $g\in \gr_{ij}$ is $L_{ij}^g=R_{\alpha}^g\times_{S^1}\mathbb{C}$ and so the groupoid product on $R_{\alpha}$ gives an isomorphism
$$L_{ij}^g\times L_{jk}^h\mapsto L_{ik}^{hg}$$
whenever $\gr_i^j \,_{t_j}\times_{s_j}\gr_j^k$ not empty and $(g,h)\in \gr_i^j \,_{t_j}\times_{s_j}\gr_j^k$. To be more precise, let $i,j,k\in I$ such that $\gr_i^j \,_{t_j}\times_{s_j}\gr_j^k$ is not empty. Consider the projections $p_1:\gr_i^j \,_{t_j}\times_{s_j}\gr_j^k\to \gr_i^j$ and 
$p_2:\gr_i^j \,_{t_j}\times_{s_j}\gr_j^k\to \gr_j^k$ in the first and in the second coordinate. Let $m:\gr_i^j \,_{t_j}\times_{s_j}\gr_j^k\to \gr_i^k$ be the restriction of the groupoid product. We have an isomorphism
\begin{equation}\label{Fellproperty}
p_1^*(L_{ij})\otimes p_2^*(L_{jk})\cong m^*(L_{ik}).
\end{equation}

The algebra of  compactly supported smooth sections of this Fell bundle, denoted by  $C_{c}^{\infty}(\gr_{\Omega },L_{\alpha_\Omega})$ has a convolution product induced by previous equation, this algebra is isomorphic to $C_{c}^{\infty}(\gr, \alpha_\Omega)$, see (23) in \cite{TXL} for an explicit isomorphism.

\begin{definition}\label{twistedkth} 
Following \cite{TXL}, we define the twisted K-theory of the twisted groupoid $(\gr,\alpha)$ by
\begin{equation}
K^i(\gr,\alpha_\Omega):=K_{-i}(C^*(\gr, \alpha_\Omega)).
\end{equation}
\end{definition}

\begin{remark}
For  the groupoid  given by  a manifold $M\rightrightarrows M$. A twisting on $M$ can be given by a Dixmier-Douday class on $H^3(M,\mathbb{Z})$. In this event, the twisted K-theory, as we defined it, coincides with  twisted K-theory defined in \cite{ASeg,Kar08}. Indeed the $C^*$-algebra $C^*(M,\alpha)$ is Morita equivalent to the continuous trace $C^*$-algebra defined by the corresponding Dixmier-Douady class (see for instance Theorem 1 in \cite{FMW}).
\end{remark}

%%%%%%%%%%%%%%%%%%%%%%%%%%%%%%%%%%%%%%%%%%%%%%%%%%%%%%%%%%%%%%%%%%%%%%%%%%%%%%%%%%%%%%%%%%%%%%%%%%%%%%%%%%%%%%%%%%%%%%%%%%%%%%%%%%%%%%%%%%%%%%%%%%%%%%%%%%%%%%%%%%%%%%%%%%%%%%%%%%%%%%%%%%%%%%%%%%%%%%%%
\section{Androulidakis-Skandalis generalized functions with pseudodifferential singularities}\label{ASsection}

In this section we want to setup the background material about pdos that we will use in the sequel. Some of the contents of this section can be found in section 1 of Androulidakis-Skandalis paper \cite{ASpdo} where the context is in principle very different.

{\bf Notation.} In the sequel, if $E$ is a smooth vector bundle we denote by $\Omega^sE$ the bundle of $s-$densities. In the case we have a given a manifold $M$ we denote by $\Omega^sM$ the bundle of $s$-densities over its tangent space. 

Let $M$ be a smooth manifold and $V$ a smooth closed submanifold of $M$. A {\it generalized function on $M$ with pseudodifferential singularity on $V$ of order $m$} is a distribution $P$ acting on $f\in C_c^\infty(M,\Omega^1(TM))$ by
$$\langle P,f\rangle:=\int_Mh(m)f(m)+(2\pi)^{-q}\int_{N^*}a(p\circ \phi(m),\xi)\chi(m)f(m)e^{-i\langle\phi(m),\xi\rangle},$$
where
\begin{itemize}
\item $p:N\to V$ is the normal bundle of $V$ in $M$, of rank $q$,
\item $\phi: U\to N$ denotes a tubular neighborhood of $V$ in $M$,
\item $h\in C^\infty(M)$,
\item $\chi$ is a smooth cut-off function equal to 1 in a neighborhood of $V$ and to 0 outside $U$, and
\item $a\in Symb_{cl}^m(N^*;\Omega^1N^*)$ is a classic symbol of order $m$ with compact support in the direction of the base\footnote{Through the paper we will assume our symbols to have compact support in the base direction}.
\end{itemize}

\begin{definition}
The generalized functions with pseudodifferential singularities of order $m$ forms a vector space denoted by $P^m(M,V)$. Those generalized functions that vanish outside a compact set of $M$ are denoted by $P_c^m(M,V)$.

In the previous construction, if $E$ is any smooth complex vector bundle over $M$ and $f\in C_c^\infty(M,\Omega^1(TM)\otimes E^*)$ we may extend easily the definition to include generalized sections of any smooth complex vector bundle $E$ over $M$, these are denoted by $P^m(M,V;E)$ (respectively $P^m_c(M,V;E)$ those with compact support).
\end{definition}

We will list below some important operations one has on the Androulidakis-Skandalis spaces:

For the spaces of generalized sections, the following properties hold

{\bf Pullback:} Consider a smooth map $g:M' \to M$ transverse to $V$, there is a functorial construction
$$g^*:P(M,V;E)\to P(M',V';g^*E)$$
for any $E$ smooth vector bundle over $M$ and $V':=g^{-1}(V)$.
We explain this in detail, let $P\in P(M,V;E)$, suppose that $P$ acts as 
$$\langle P,f\rangle:=\int_Mh(m)f(m)+(2\pi)^{-q}\int_{N^*}a(p\circ \phi(m),\xi)\chi(m)f(m)e^{-i\langle\phi(m),\xi\rangle},$$
for $f\in C_c^\infty(M,\Omega^1(TM)\otimes E^*)$, $h\in C^\infty(M,E)$, $a\in Symb_{cl}^m(N^*;\Omega^1N^*\otimes E)$ and $\phi$ and $\chi$ as above.
We want to explicitly describe, for $f'\in C_c^\infty(M',\Omega^1(TM')\otimes (g^*E)^*)$ 
$$\langle g^*P,f\rangle.$$
For this it will be enough to describe the corresponding structure data $h',a', \phi'$ and  $\chi'$ associated to $h, a, \phi$ and $\chi$:
\begin{itemize}
\item We let $h'=h\circ g$.
\item Next, by the transversality assumption we have that the derivative in the normal direction of $g$ induces an injective morphism of bundles
\begin{equation}\label{pullconstruction}
g^*N^*\stackrel{(d_Ng)^*}{\longrightarrow}N'^*
\end{equation}
that allows to identify $g^*N^*$ with a subbundle of $N'^*$, where we are denoting by $N'$ the normal bunlde of $V'$ in $M'$. We can let first $g^*a\in Symb^m_{cl}(g^*N^*,\Omega^1(g^*N^*)\otimes g^*E)$ the induced symbol  by $a$ (which always exists) and extend it by zero to a symbol $a'\in Symb^m_{cl}(N^*,\Omega^1(N'^*)\otimes g^*E)$, we will argue below why the operator $g^*P$ wont depend on this extension.
\item To continue, consider a tubular neighborhood for $V'$ in $M'$, $\phi':U'\to N'$, compatible with $\phi$ in the sense that
$\phi\circ g=d_Ng\circ \phi'$.
\item To finish, consider a smooth cut off function $\chi':M'\to \R$ equal to 1 in $V'$ and zero outside $U'$ compatible with $\chi$, {\it i.e.,} $\chi'=\chi\circ g$.
\end{itemize}
It is immediate to check that the operator on $P(M',V';g^*E)$ associated to the above data does not depend by definition on the extension of the tubular neigborhood, of the symbol and of the cut off function, we denote it by $g^*P$. Also, by construction of $h',a', \phi'$ and  $\chi'$ above it is a direct computation to check the following proposition.

%The key observation is to note that the transversality hypothesis implies that the map $g^*:N^*(M,V)\to N^*(M',V')$ induced by $g$ is a submersion and hence it defines, by integration along the fibers, a morphism
%$${\tiny Symb_{cl}^m(N^*(M,V);\Omega^1N^*(M,V)\otimes E)\stackrel{\widetilde{g}}{\rightarrow}Symb_{cl}^m(N^*(M',V');\Omega^1N(M',V')^*\otimes g^*E)}.$$
%We have, by definition,
%$$\langle g^*P,f\rangle:=\int_{M'}h(g(x))k(x)+(2\pi)^{-q}\int_{N^*(M',V')}\widetilde{g}(a)(p'\circ \phi'(x),\xi)\chi'(x)k(x)e^{-i\langle\phi'(x),\xi\rangle}.$$
%where $\phi':U'\to N'$ is a tubular neighborhood map compatible with $\phi$ and  
\begin{proposition}[Pullback functoriality]\label{Pullfunct}
The main property of this construction is the functoriality, meaning that for $M''\stackrel{g'}{\rightarrow}M'\stackrel{g}{\rightarrow}M$ satisfying the hypothesis above we have
$$g'^*\circ g^*=(g\circ g')^*.$$
\end{proposition}
%\begin{proof}\end{proof}

{\bf Pushforward:} Let $p:M'\to M$ be a smooth submersion such that $p$ restricts to a submersion $p:V'\to V$, then there is a functorial construction
$$p!:P(M',V';\Omega^1Ker\,dp\otimes p^*E)\to P(M,V;E)$$
for any $E$ smooth vector bundle over $M$.
Let $P'\in P(M',V';\Omega^1Ker\,dp\otimes p^*E)$, suppose that $P'$ acts as 
$$\langle P',f'\rangle:=\int_{M'}h'(x)f'(x)+(2\pi)^{-q}\int_{N'^*}a'(p'\circ \phi'(x),\xi)\chi'(x)f'(x)e^{-i\langle\phi'(x),\xi\rangle},$$ 
for $f'\in C_c^\infty(M',\Omega^1(TM')\otimes (\Omega^1Ker\,dp\otimes p^*E)^*)$, $h'\in C^\infty(M',\Omega^1Ker\,dp\otimes p^*E)$, $a'\in Symb_{cl}^m(N'^*;\Omega^1N'^*\otimes \Omega^1Ker\,dp\otimes p^*E)$ and $\phi'$ and $\chi'$ as in the definition.
For $f\in C_c^\infty(M,\Omega^1(TM)\otimes E^*)$ we want to explicitly describe, 
$$\langle p!P,f\rangle.$$
As for the pullback case we will describe the corresponding structure data $h,a, \phi$ and  $\chi$ associated to $h', a', \phi'$ and $\chi'$:
\begin{itemize}
\item The section $h$ is obtained by integrating $h'$ along he fibers of $p$, {\it i.e.},  $h(m):=\int_{m'\in p^{-1}(m)}h'(m')$.
\item We have again that $p$ induces an injection 
$N^*\to N'^*$, because $p$ is a submersion. Now, because $p|_{V'}$ is also a submersion and the symbol $a'$ has compact support in the $V'$-direction, we have that
$$a(v,\xi):=\int_{w\in p^{-1}(v)}a'(w,(d_Np)^*(\xi))$$
defines a symbol in $Symb_{cl}^m(N^*;\Omega^1N^*\otimes\otimes E)$
\item Given $\phi':U'\to N'$, since $p|_{U'}$ and $d_Np:N'\to N$ are submersions, we can construct a tubular neighborhood 
$\phi:U\to N$ of $V$ in $M$ compatible with $\phi'$ ({\it i.e.}, $d_Np\circ \phi'=\phi\circ p$), essentially by projecting $\phi'$.
\item Finally, as above, we consider a compatible cut off function $\chi'$.
\end{itemize}
Again, it is an exercise to check that we obtain an operator $p!P$ that does not depend on the way the $\phi$ and $\chi$ are constructed. The following proposition is a direct computation of the definition above:
\begin{proposition}[Pushfoward functoriality]\label{Pushfunct}
The main property of this construction is the functoriality, meaning that for $M''\stackrel{p}{\rightarrow}M'\stackrel{q}{\rightarrow}M$ satisfying the hypothesis above we have
$$q!\circ p!=(q\circ p)!.$$
\end{proposition}
%\begin{proof}
%\end{proof}

{\bf Compatibility between the pullback and the pushfoward:} We will prove below a compatibility result result between the two operations we have just defined. %This is the main technical result in order to get in the next section a well defined associative product in our algebra of twisted pseudodifferential operators. 
We enounce the proposition.

\begin{proposition}[Pullback vs. Pushforward]\label{PullPush}
Given a commutative diagram
\begin{equation}\label{pullpushmaps}
\xymatrix{
X'\ar[r]^-g\ar[d]_-q&M'\ar[d]^-p\\
X\ar[r]_-f&M
}
\end{equation}
of smooth maps with $f,g$ satisfying the hypothesis of proposition \ref{Pullfunct} and $p,q$ satisfying the hypothesis of proposition \ref{Pushfunct}, the following diagram is commutative
\begin{equation}\label{pullpushfunct}
\xymatrix{
P^*_c(M',V';p^*E\otimes \Omega^1Ker\, dp)\ar[r]^-{g^*}\ar[d]_-{p!}&
P^*_c(X',g^{-1}(V');g^*(p^*E\otimes \Omega^1Ker\, dp))\ar[d]^-{q!}\\
P^*_c(M,V;E)\ar[r]_-{f^*}&P^*_c(X,f^{-1}(V);f^*E)
}
\end{equation}
That is, $f^*\circ p!=q!\circ g^*$.
\end{proposition}
\begin{proof}
Essentially it resumes to check that the structure data $(a,h,\phi,\chi)$ used to construct an operator satisfies the above compatibility property. Let us check this first for a section $h\in C_c^\infty(M',p^*E\otimes \Omega^1\,Ker\,dp)$, by applying the first the pullback construction for $g$ an then pushforward construction for $q$, to $h$ corresponds the section in $C_c^\infty(X,f^*E)$ given explicitly by
\begin{equation}\label{hpullpush}
x\mapsto \int_{x'\in q^{-1}(x)}(h\circ g)(x'),
\end{equation}
in the other hand, by applying first the pushforward construction for $p$ an then pullback construction for $f$, to $h$ corresponds the section in $C_c^\infty(X,f^*E)$ given explicitly by
\begin{equation}\label{hpushpull}
x\mapsto \int_{m'\in p^{-1}(f(x))}h(m').
\end{equation} 
The diagram (\ref{pullpushmaps}) being commutative implies immediately that these two sections coincide.
For the symbol $a$, one has, by applying the first the pullback construction for $g$ an then pushforward construction for $q$, to $a$ corresponds the symbol in $Symb_{cl}^*(N^*(X,f^{-1}(V));\Omega^1(N^*(X,f^{-1}(V)))\otimes f^*E)$ given explicitly by
\begin{equation}\label{apullpush}
(x,\xi)\mapsto \int_{x'\in q^{-1}(x)}\widetilde{a}(x',(d_Nq)^*(\xi)),
\end{equation}
where $\widetilde{a}$ is the extension of $g^*(a)$ to $N^*(X',g^{-1}(V))$. In the other hand, by applying first the pushforward construction for $p$ an then pullback construction for $f$, to $a$ corresponds the symbol in $Symb_{cl}^*(N^*(X,f^{-1}(V));\Omega^1(N^*(X,f^{-1}(V)))\otimes f^*E)$ given as the extension of
\begin{equation}\label{apushpull}
(m,\eta)\mapsto \int_{m'\in p^{-1}(m)}a(m',(d_Np)^*(\eta)),
\end{equation} 
which is a symbol on $N^*(M,V)$, to $N^*(X,f^{-1}(V))$, using the transverse map $f$ as (\ref{pullconstruction}) above. The commutativity of the diagram (\ref{pullpushmaps}) implies again these two symbols are the same. 
Finally, tubular neighborhoods, and cut off functions, can be easily constructed in order to have the required compatibility.
\end{proof}

{\bf External Product:} Consider, for $i=1,2$, a couple of submersions $M\stackrel{p_i}{\longrightarrow}M_i$ together with a couple of submanifolds $V_i\subset M_i$ satisfying the condition for the pullback construction above. The following proposition is stated and proven in \cite{ASpdo} proposition 1.10.
%Consider a commutative diagram of submersions of the following form
%$$ \xymatrix{ M\ar[d]_-{p_2}\ar[r]^-{p_1}&M_1\ar[d]^-{\pi_1}\\ M_2\ar[r]_-{\pi_2}&B }$$
\begin{proposition}[External Product]\label{ASproduct}
Let, for $i=1,2$, $K_i\in P_c^{m_i}(M_i,V_i;E_i)$. Consider $Q_i:=p_i^*K_i\in P_c^{m_i}(M,W_i;p_i^*(E_i))$ where $W_i:=p_i^{-1}(V_i)$. We have
\begin{itemize}
\item[(a)] The product $Q_1\cdot Q_2$ makes sense as a distribution acting on $C_c^\infty(M,\Omega^1M\otimes E^*)$ where $E=p_1^*(E_1)\otimes p_2^*(E_2)$.
\item[(b)] Suppose besides there is a submersion $m:M\to N$ strictly transverse to both $W_1$ and $W_2$ and with 
$p_1^*(E_1)\otimes p_2^*(E_2)=m^*(F)\otimes \Omega^1Ker\, dm$ for some bundle $F$, then we can perform the pushforward map construction to obtain a pseudodifferential distribution
$$m!(Q_1\cdot Q_2)\in P_c^{m_1+m_2}(N,m(W_1\cap W_2);F).$$
\end{itemize}
\end{proposition}

%\begin{equation} \xymatrix{ &P_c^m(M_1,V_1;E_1)\times P_c^n(M_2,V_2;E_2)\ar[d]&\\ &P_c^{m+n}(M,V;E)&} \end{equation}

%\section{Pseudodifferential calculus for Lie groupoids}

\section{Projective pseudodifferential calculus for Lie groupoids}\label{pdogrpds}

Let $\gr\rightrightarrows M$ be a Lie groupoid and let $\alpha$ be a twisting. Consider a cocycle
$$\alpha_\Omega\to PU(H)$$
associated to an open cover $\Omega=\{\Omega_i\}_{i\in M}$ of $M$
representing $\alpha$.
For every $(i,j)\in I^2$ there is a bisubmersion
$(\gr_i^j,s,t)$, and in fact $\{(\gr_i^j,s,t)\}_{(i,j)}$ is an atlas of bisubmersions adapted to $\{(\gr,s,t)\}$. For each $(i,j)\in I^2$ we have the line bundle $L_{ij}\to \gr_i^j$ associated to $\alpha$, we consider the space of compactly supported generalized 
sections on $\gr_i^j$ with pseudodifferential singularities on $\Omega_{ij}$ of order $m$, denoted by $$P_c^m(\gr_i^j,\Omega_{ij};\Omega^{\frac{1}{2}}\otimes L_{ij}),$$
where by $\Omega^{\frac{1}{2}}$ we mean the bundle of half densities over $\gr_i^j$,
$$\Omega^{\frac{1}{2}}:=\Omega^{\frac{1}{2}}(Ker\,ds_j\oplus Ker\, dt_i),$$
where $s_j:\gr_i^j\to \Omega_j$ and $t_i:\gr_i^j\to \Omega_i$ are the source and target maps respectively. We wont add $i,j$ to the notation for $\Omega^{\frac{1}{2}}$ since it will be clear from each particular context.
%By proposition 1.4 in \cite{ASpdo} we have an exact sequence
%\begin{equation}\small{\xymatrix{ 0\ar[r]&P_c^{m-1}(\gr_i^j,\Omega_{ij};\Omega^{\frac{1}{2}}\otimes L_{ij})\ar[r]& P_c^m(\gr_i^j,\Omega_{ij};\Omega^{\frac{1}{2}}\otimes L_{ij})\ar[r]&C^\infty(S^*N_{ij},End(L_{ij}))\ar[r]&0 }} \end{equation}
Using the Fell bundle structure of the line bundle $L$ over $\gr_\Omega$ we will define a product
\begin{equation}
\xymatrix{
&P_c^m(\gr_i^j,\Omega_{ij};\Omega^{\frac{1}{2}}\otimes L_{ij})\times P_c^n(\gr_j^k,\Omega_{jk};\Omega^{\frac{1}{2}}\otimes L_{jk})\ar[d]&\\
&P_c^{m+n}(\gr_i^k,\Omega_{ik};\Omega^{\frac{1}{2}}\otimes L_{ik})&
}
\end{equation}
We need the following technical lemma which is essentially proved in \cite{LMV} lemma 12.
\begin{lemma}\label{LMVlemma}
Let $i,j,k\in I$ such that $\gr_i^j \,_{t_j}\times_{s_j}\gr_j^k$ is not empty. Consider the projections $p_1:\gr_i^j \,_{t_j}\times_{s_j}\gr_j^k\to \gr_i^j$ and 
$p_2:\gr_i^j \,_{t_j}\times_{s_j}\gr_j^k\to \gr_j^k$ in the first and in the second coordinate. Let $m:\gr_i^j \,_{t_j}\times_{s_j}\gr_j^k\to \gr_i^k$ be the restriction of the groupoid product. We have an isomorphism 
\begin{equation}
p_1^*(\Omega^{\frac{1}{2}}\otimes L_{ij})\otimes p_2^*(\Omega^{\frac{1}{2}}\otimes L_{jk})\cong \Omega^1Ker\,dm\otimes m^*(\Omega^{\frac{1}{2}}\otimes L_{ik})
\end{equation}
\end{lemma}
\begin{proof}
In fact we have 
\begin{equation}
p_1^*(\Omega^{\frac{1}{2}}\otimes L_{ij})\otimes p_2^*(\Omega^{\frac{1}{2}}\otimes L_{jk})\cong p_1^*(\Omega^{\frac{1}{2}})\otimes p_2^*(\Omega^{\frac{1}{2}})\otimes (p_1^*(L_{ij})\otimes p_2^*(L_{jk}))
\end{equation}
and by lemma 12 in \cite{LMV} one has
\begin{equation}
p_1^*(\Omega^{\frac{1}{2}})\otimes p_2^*(\Omega^{\frac{1}{2}})\cong \Omega^1Ker\,dm\otimes m^*(\Omega^{\frac{1}{2}})
\end{equation}
and so we conclude by using the Fell property (\ref{Fellproperty}):
\begin{equation}
p_1^*(L_{ij})\otimes p_2^*(L_{jk})\cong m^*(L_{ik}).
\end{equation}
\end{proof}

We can now state the product announced above:

Let $K_{ij}\in P_c^m(\gr_i^j,\Omega_{ij};\Omega^{\frac{1}{2}}\otimes L_{ij})$ and $K_{jk}\in P_c^n(\gr_j^k,\Omega_{jk};\Omega^{\frac{1}{2}}\otimes L_{jk})$. By proposition \ref{ASproduct} (a) the product
$$p_1^*(K_{ij})\cdot p_2^*(K_{jk})$$ makes sense as a distribution acting on  
$$C_c^\infty(\gr_i^j\,_{t_j}\times_{s_j}\gr_j^k,\Omega^1(\gr_i^j\,_{t_j}\times_{s_j}\gr_j^k)\otimes (p_1^*(\Omega^{\frac{1}{2}}\otimes L_{ij})\otimes p_2^*(\Omega^{\frac{1}{2}}\otimes L_{jk}))^*)$$
which is, after the lemma above, isomorphic to 
$$C_c^\infty(\gr_i^j\,_{t_j}\times_{s_j}\gr_j^k,\Omega^1(\gr_i^j\,_{t_j}\times_{s_j}\gr_j^k)\otimes(\Omega^1Ker\,dm\otimes m^*(\Omega^{\frac{1}{2}}\otimes L_{ik}))^*).$$ 
We can now apply the external product construction \ref{ASproduct} (b) to define 
\begin{equation}\label{twistdistproduct}
K_{ij}*K_{jk}:=m!(p_1^*(K_{ij})\cdot p_2^*(K_{jk}))\in P_c^{m+n}(\gr_i^k,\Omega_{ik};\Omega^{\frac{1}{2}}\otimes L_{ik})
\end{equation}

We will also need to introduce an involution on the algebra of pseudodifferential operators below. For this, denote by $\iota:\gr_\Omega\rightarrow \gr_\Omega$ the Lie groupoid inversion.
Let $K_{ij}\in P_c^m(\gr_i^j,\Omega_{ij};\Omega^{\frac{1}{2}}\otimes L_{ij})$, since $\iota^*(L_{ji}\otimes \Omega^{\frac{1}{2}})\cong L_{ij}\otimes \Omega^{\frac{1}{2}}$ we have that $\iota!$ sends $P_c^m(\gr_i^j,\Omega_{ij};\Omega^{\frac{1}{2}}\otimes L_{ij})$ in $P_c^m(\gr_j^i,\Omega_{ji};\Omega^{\frac{1}{2}}\otimes L_{ji})$, we let
\begin{equation}\label{twistdistinvolution}
K_{ij}^*:=\overline{\iota!(K_{ij})}.
\end{equation}

The following is one of the main results of this paper.
\begin{theorem}[Algebra of twisted distributions]\label{pdoalg}
Let us denote, for each $m\in \mathbb{Z}$, $$P_c^m(G,\alpha_\Omega):=
\bigoplus_{(i,j)}P_c^m(\gr_i^j,\Omega_{ij};\Omega^{\frac{1}{2}}\otimes L_{ij}).$$
With the product (\ref{twistdistproduct}) and the involution (\ref{twistdistinvolution}) described above, the union $$P_c^\infty(\gr,\alpha_\Omega)=\bigcup_{m\in \mathbb{Z}}P_c^m(\gr,\alpha_\Omega)$$
forms a filtered $*-$algebra with bilateral ideal 
$$P_c^{-\infty}(\gr,\alpha_\Omega)=\bigcap_{m\in \mathbb{Z}}P_c^m(\gr,\alpha_\Omega).$$
\end{theorem}
\begin{proof}
We will concentre in the proof of the associativity of the product, the rest of the properties being immediate. As we remarked above, associativity is not a direct consequence of the external product associativity. Let $(i,j,k,l)\in I^4$ such that $\gr_i^j\,_{t_j}\times_{s_j} \gr_j^k\,_{t_k}\times_{s_k} \gr_k^l\neq \emptyset$ and consider $P_{ij}\in P_c^*(\gr_i^j,\Omega_{ij};\Omega^{\frac{1}{2}}\otimes L_{ij}), P_{jk}\in P_c^*(\gr_j^k,\Omega_{jk};\Omega^{\frac{1}{2}}\otimes L_{jk}, P_{kl}\in P_c^*(\gr_k^l,\Omega_{kl};\Omega^{\frac{1}{2}}\otimes L_{kl})$. We will prove that
$$(P_{ij}*P_{jk})*P_{kl}=P_{ij}*(P_{jk}*P_{kl})$$
by using propositions \ref{Pullfunct}, \ref{Pushfunct} and \ref{PullPush} above.

By definition
\begin{equation}\label{eq1asso}
\begin{array}{ccc}
(P_{ij}*P_{jk})*P_{kl}&=&m!(p_1^*(P_{ij}*P_{jk})\cdot p_2^*(P_{kl}))\\
&=&m!(p_1^*(m!(p_1^*(P_{ij})\cdot p_2^*(P_{jk})))\cdot p_2^*(P_{kl}))\\
&=&m!((p_1^*\circ m!)(p_1^*(P_{ij})\cdot p_2^*(P_{jk}))\cdot p_2^*(P_{kl})).
\end{array}
\end{equation}
Now, consider the following commutative diagram
\begin{equation}
\xymatrix{
\gr_i^j\,_{t_j}\times_{s_j}\gr_j^k\,_{t_k}\times_{s_k}\gr_k^l
\ar[d]_-{(p_1,p_2)}\ar[rr]^-{(m,1)}&&\gr_i^k\,_{t_k}\times_{s_k}\gr_k^l\ar[d]^-{p_1}\\
\gr_i^j\,_{t_j}\times_{s_j}\gr_j^k\ar[rr]_-{m}&&\gr_i^k
}
\end{equation}
which satisfies the conditions of proposition \ref{PullPush}. In particular,
$$p_1^*\circ m!=(m,1)!\circ (p_1,p_2)^*,$$
and so, from (\ref{eq1asso}),
\begin{equation}\label{eq2asso}
(P_{ij}*P_{jk})*P_{kl}=m!(((m,1)!\circ (p_1,p_2)^*)(p_1^*(P_{ij})\cdot p_2^*(P_{jk}))\cdot p_2^*(P_{kl}))
\end{equation}
Now, by proposition \ref{Pullfunct} applied to 
\begin{equation}
\gr_i^j\,_{t_j}\times_{s_j}\gr_j^k\,_{t_k}\times_{s_k}\gr_k^l
\stackrel{(p_1,p_2)}{\longrightarrow} \gr_i^j\,_{t_j}\times_{s_j}\gr_j^k\stackrel{p_1}{\longrightarrow} \gr_i^j
\end{equation}
and to 
\begin{equation}
\gr_i^j\,_{t_j}\times_{s_j}\gr_j^k\,_{t_k}\times_{s_k}\gr_k^l\stackrel{(p_1,p_2)}{\longrightarrow} \gr_i^j\,_{t_j}\times_{s_j}\gr_j^k\stackrel{p_2}{\longrightarrow}\gr_j^k 
\end{equation}
we obtain 
$$(p_1,p_2)^*(p_1^*(P_{ij})\cdot p_2^*(P_{jk}))=p_1^*(P_{ij})\cdot p_2^*(P_{jk})$$
and 
$$((m,1)!\circ (p_1,p_2)^*)(p_1^*(P_{ij})\cdot p_2^*(P_{jk}))\cdot p_2^*(P_{kl})=(m,1)!(p_1^*(P_{ij})\cdot p_2^*(P_{jk})\cdot p_3^*(P_{kl}))$$
by definition of $(m,1)$ and the associativity of the external product. Hence, from (\ref{eq2asso}), we get
\begin{equation}\label{eq3asso}
(P_{ij}*P_{jk})*P_{kl}=m!((m,1)!(p_1^*(P_{ij})\cdot p_2^*(P_{jk})\cdot p_3^*(P_{kl})))
\end{equation}
which by proposition \ref{Pushfunct} applied to the obvious groupoid product associative diagram gives
\begin{equation}\label{eq4asso}
(P_{ij}*P_{jk})*P_{kl}=(1,m)!(m!(p_1^*(P_{ij})\cdot p_2^*(P_{jk})\cdot p_3^*(P_{kl})))
\end{equation}
Now, we have, by applying the backwards arguments, that 
\begin{equation}\label{eq5asso}
(1,m)!(m!(p_1^*(P_{ij})\cdot p_2^*(P_{jk})\cdot p_3^*(P_{kl})))=P_{ij}*(P_{jk}*P_{kl})
\end{equation}
which ends the proof.
\end{proof}

\begin{definition}
The space $P_c^\infty(\gr,\alpha_\Omega)$ will be called the algebra of twisted pseudodifferential distributions (with respect to $\alpha_\Omega$ if needed). We will refer to the elements in the ideal $P_c^{-\infty}(\gr,\alpha_\Omega)$ as twisted regularizing distributions.
\end{definition}

\begin{remark}\label{twistedsmoothingop}
In fact, one of the advantages of having used Androulidakis-Skandalis definition is that we immediately have an identification
\begin{equation}
C_c^\infty(\gr,\alpha_\Omega)\cong P_c^{-\infty}(\gr,\alpha_\Omega).
\end{equation}
\end{remark}

The last proposition allow to realize, as in the untwisted case, the pseudodifferential distributions as multipliers, we resume this in the following corollary, whose proof follows immediately from proposition 14 in \cite{LMV}.

%Let $K_{ij}\in P_c^m(\gr_i^j,\Omega_{ij};\Omega^{\frac{1}{2}}\otimes L_{ij})$, it defines by convolution an operator on $C_c^\infty(\gr,\alpha_\Omega)$ in the following way: Let $f \in C_c^\infty(\gr,\alpha_\Omega)$, we will define the convolution componentwise, denote by $f_{jk}$ the component of $f$ in $\gr_j^k$, we let $$K_{ij}*f_{jk}\in C_c^\infty(\gr_i^k,\Omega^{\frac{1}{2}}\otimes L_{ik})$$ be defined by \begin{equation}\label{twistdist} (K_{ij}*f_{jk})(\gamma):=\langle K_{ij},  f_{jk}(\gamma \cdot (\,)^{-1})\rangle.\end{equation}
%The proof that the formula above gives indeed an element in $C_c^\infty(\gr_i^k,\Omega^{\frac{1}{2}}\otimes L_{ik})$ follows exactly as the proof of proposition 14 in \cite{LMV}.

%Notice that $K_{ij}$ interchanges components, sending the $(j,k)$ component to the $(i,k)$ component whenever $\gr_i^j\,_{t_j}\times_{s_j}\gr_j^k$ is not empty. 

\begin{corollary}\label{smoothmultipliers}
With the notations above, we have
\begin{enumerate}
\item Every $K\in P_c^\infty(\gr,\alpha_\Omega)$ defines by convolution a (left) multiplier of $C_c^\infty(\gr,\alpha_\Omega)$.
\item The map $K\mapsto K*(\cdot)$ defines a monomorphism of algebras 
$$P_c^\infty(\gr,\alpha_\Omega) \to M(C_c^\infty(\gr,\alpha_\Omega)).$$
\end{enumerate}
\end{corollary}
%\begin{proof}
%\end{proof}
%\begin{proof} \begin{enumerate} \item By definition of generalized sections with pseudodifferential singularities one has that $C_c^\infty(\gr_i^j,\Omega^{\frac{1}{2}}\otimes L_{ij})\subset P_c^m(\gr_i^j,\Omega_{ij};\Omega^{\frac{1}{2}}\otimes L_{ij})$ for every $m$ and for every $(i,j)\in I$. So if we take $K_{ij}\in P_c^m(\gr_i^j,\Omega_{ij};\Omega^{\frac{1}{2}}\otimes L_{ij})$ and $f \in C_c^\infty(\gr,\alpha_\Omega)$, by the precedent proposition, the convolution $$K_{ij}*f$$ makes sense as a distribution in $P^{-\infty}(\gr,\alpha_\Omega)\cong C_c^\infty(\gr,\alpha_\Omega)$. 
 
%We will now justify that this is indeed a compactly supported function. For this it is enough to do it in each component of $f$, for example take $f_{jk}$ the component of $f$ in $\gr_j^k$. Let $g\in C^\infty(\gr_i^k,\Omega^1T(\gr_i^k)\otimes L_{ik}^*)$ be a test function, by definition $$<K_{ij}*f_{jk},g>=\int \langle K_{ij},  f_{jk}(\gamma \cdot (\,)^{-1})\rangle g(\delta)$$ \end{enumerate} \end{proof}

Denote by $\Psi_c^m(\gr,\alpha_\Omega)$ the image in the multiplier algebra $M(C_c^\infty(\gr,\alpha_\Omega))$ of $P_c^m(\gr,\alpha_\Omega)$. The union 
$$\Psi_c^\infty(\gr,\alpha_\Omega)=\bigcup_{m\in \mathbb{Z}}\Psi_c^m(\gr,\alpha_\Omega)$$
forms a filtered subalgebra of $M(C_c^\infty(\gr,\alpha_\Omega))$ with bilateral ideal 
$$\Psi_c^{-\infty}(\gr,\alpha_\Omega)=\bigcap_{m\in \mathbb{Z}}\Psi_c^m(\gr,\alpha_\Omega)$$

\begin{definition}
The algebra $\Psi_c^{\infty}(\gr,\alpha_\Omega)$ is called the algebra of projective pseudodifferential operators associated to $\alpha_\Omega$. The elements in $\Psi_c^{-\infty}(\gr,\alpha_\Omega)$ are called as usual the regularizing operators.
\end{definition}

\subsection{Projective operators as twisted invariant families}\label{twistedfamiliessection}

Every $K\in P_c^\infty(\gr,\alpha_\Omega)$ defines by convolution an operator
$$P_K:C_c^\infty(\gr,\alpha_\Omega)\to C_c^\infty(\gr,\alpha_\Omega).$$
By definition of the convolution product of distributions it has the property that if $K_{ij}$ denotes the component of $K$ corresponding to $\gr_i^j$ and if $k\in I$ with $\gr_i^j\,_{t_j}\times_{s_j}\gr_j^k$ is not empty the operator sends
$$P_{K_{ij}}^k:C_c^\infty(\gr_j^k,\Omega^{\frac{1}{2}}\otimes L_{jk})\to C_c^\infty(\gr_i^k,\Omega^{\frac{1}{2}}\otimes L_{ik}).$$
Moreover we will see that it also satisfies some invariance property when one changes of $k$. To be more precise, for $r\in R_\alpha$ of the form $r=((l,g,k),u)$ with $s(g)=x$ and $t(g)=y$ one has an isomorphism
$$U_j^r:C^\infty((\gr_j^k)^x,\Omega^{\frac{1}{2}}\otimes L_{jk})\to C^\infty((\gr_j^l)^y,\Omega^{\frac{1}{2}}\otimes L_{jl}),$$
for every $j\in I$ given as follows: Consider $[(r,1)]\in L_{kl}^g$, then the Fell product on $L$ implies that $[(r,1)]$ defines by multiplication an isomorphism 
$L_{jk}^{g^{-1}\eta}\to L_{ik}^{\eta}$ for every $\eta \in (\gr_j^l)^y$, given $v\in L_{jk}^{g^{-1}\eta}$ we denote by $r\cdot v\in L_{ik}^{\eta}$ this action. The morphism $U_j^r$ is given by
$$U_j^r(f)(\eta):=r\cdot f(g^{-1}\eta),$$
and it is an isomorphism with inverse $U_j^{r^{-1}}$.

It should be by now not very surprising that some invariance property hold, indeed in the (untwisted) Lie groupoid case it has been known by experts that the pseudodifferential operators might be defined by global distributions on the groupoid and that the invariance property is just a natural consequence of the convolution product of distributions. In fact it was very recently in \cite{LMV} that Lescure, Manchon and Vassout formalized these ideas and went further in the study of groupoid convolutions. Our result is the following:

\begin{proposition}
Let $K\in P_c^\infty(\gr,\alpha_\Omega)$ and $P_K$ the associated pseudodifferential operator as above. The following diagram is commutative
\begin{equation}
\xymatrix{
C_c^\infty((\gr_j^k)^x,\Omega^{\frac{1}{2}}\otimes L_{jk})\ar[d]_-{U_j^r}
\ar[rr]^-{P_{ij,x}^k}&&C_c^\infty((\gr_i^k)^x,\Omega^{\frac{1}{2}}\otimes L_{ik})\ar[d]^-{U_i^r}\\
C_c^\infty((\gr_j^l)^y,\Omega^{\frac{1}{2}}\otimes L_{jl})\ar[rr]_-{P_{ij,y}^l}&&C_c^\infty((\gr_i^l)^y,\Omega^{\frac{1}{2}}\otimes L_{il})
}
\end{equation}
for every $r\in R_\alpha$ of the form $r=((l,g,k),u)$ with $s(g)=x$, $t(g)=y$ and where $P_{ij,x}^k:=(P_{K_{ij}}^k)|_{C_c^\infty((\gr_j^k)^x,\Omega^{\frac{1}{2}}\otimes L_{jk})}$ and $P_{ij,y}^l:=(P_{K_{ij}}^l)|_{C_c^\infty((\gr_j^l)^y,\Omega^{\frac{1}{2}}\otimes L_{jl})}$.
\end{proposition}

\begin{proof}
Let $K_{ij}\in P_c^\infty(\gr_i^j,\Omega^{\frac{1}{2}}\otimes L_{ij})$ and 
$f\in C_c^\infty(\gr_j^k,\Omega^{\frac{1}{2}}\otimes L_{jk})$, we want to prove the following equality
\begin{equation}\label{twistinvdist}
K_{ij}*U_j^r(f)=U_i^r(K_{ij}*f)
\end{equation}
By definition of the convolution product it will follow from the following three facts

{\bf I. } Compatibility between the pullback $p_2^*$ and the twisted action: That is, the following diagram is commutative
\begin{equation}
\xymatrix{
C_c^\infty(\gr_j^k,\Omega^{\frac{1}{2}}\otimes L_{jk})
\ar[d]_-{U_j^r}\ar[r]^-{p_2^*}&C_c^\infty(\gr_i^j\,_{t_j}\times_{s_j}\gr_j^k,p_2^*(\Omega^{\frac{1}{2}}\otimes L_{jk}))\ar[d]^-{\widetilde{U_j^r}}\\
C_c^\infty(\gr_j^l,\Omega^{\frac{1}{2}}\otimes L_{jl})\ar[r]_-{p_2^*}
%\ar[d]^-{\widetilde{U_j^r}}
&C_c^\infty(\gr_i^j\times_{\Omega_j}\gr_j^l,p_2^*(\Omega^{\frac{1}{2}}\otimes L_{jl}))
}
\end{equation}
where $\widetilde{U_j^r}$ is defined as $U_j^r$ but just in the second coordinate.

{\bf II.} The following equality holds since the $r$-action is only on the second coordinate:
\begin{equation}
p_1^*(K_{ij})\cdot \widetilde{U_j^r}(p_2^*(f))=\widetilde{U_j^r}(p_1^*(K_{ij})\cdot p_2^*(f)).
\end{equation}

{\bf III.} Compatibility between the pushforward $m!$ and the twisted action: Indeed, a direct computation shows that
\begin{equation}
m!(\widetilde{U_j^r}(p_1^*(K_{ij})\cdot p_2^*(f)))=U_i^r(m!(p_1^*(K_{ij})\cdot p_2^*(f)))).
\end{equation}
We can check now (\ref{twistinvdist}) using {\bf I,II,III} above:
\begin{equation}\nonumber
K_{ij}*U_j^r(f)=
m!(p_1^*(K_{ij})\cdot p_2^*(U_j^r(f)))=m!(p_1^*(K_{ij})\cdot \widetilde{U_j^r}(p_2^*(f)))
\end{equation}
\begin{equation}\nonumber
=m!(\widetilde{U_j^r}(p_1^*(K_{ij})\cdot p_2^*(f)))=U_i^r(K_{ij}*f)
\end{equation}

%the following diagram is commutative
%\begin{equation}
%\xymatrix{
%\ar[d]_-{\widetilde{U_j^r}}\ar[r]^-{m!}&\ar[d]^-{U_i^r}\\
%\ar[r]_-{m!}
%\ar[d]^-{\widetilde{U_j^r}} &}\end{equation}
%It is based on the following three 
\end{proof}

The last proposition suggests the definition of operators invariant under a twisted action (or an action of the associated extension) induced from the twisting.

\begin{definition}
An $\alpha_\Omega$-operator is an operator
$$P:C_c^\infty(\gr,\alpha_\Omega)\to C^\infty(\gr,\alpha_\Omega).$$
given by a family of operators
$$P_{ij,x}^k:C_c^\infty((\gr_j^k)^x,\Omega^{\frac{1}{2}}\otimes L_{jk})\to C^\infty((\gr_i^k)^x,\Omega^{\frac{1}{2}}\otimes L_{ik})$$
in the sense that
$$P(f_{jk})(\gamma)=P_{ij,x}^k(f_{jk}|_{(\gr_j^k)^x})(\gamma)$$
for $\gamma\in (\gr_i^k)^x$ and $f_{jk}\in C_c^\infty(\gr_j^k,\Omega^{\frac{1}{2}}\otimes L_{jk})$; and such that the following diagram
\begin{equation}
\xymatrix{
C_c^\infty((\gr_j^k)^x,\Omega^{\frac{1}{2}}\otimes L_{jk})\ar[d]_-{U_j^r}
\ar[rr]^-{P_{ij,x}^k}&&C_c^\infty((\gr_i^k)^x,\Omega^{\frac{1}{2}}\otimes L_{ik})\ar[d]^-{U_i^r}\\
C_c^\infty((\gr_j^l)^y,\Omega^{\frac{1}{2}}\otimes L_{jl})\ar[rr]_-{P_{ij,y}^l}&&C_c^\infty((\gr_i^l)^y,\Omega^{\frac{1}{2}}\otimes L_{il})
}
\end{equation}
is commutative for every $r\in R_\alpha$ of the form $r=((l,g,k),u)$ with $s(g)=x$, $t(g)=y$.
\end{definition}

\subsection{Projective operators acting on twisted vector bundles}

In practice one would like to admit operators (or distributions) acting on sections of vector bundles. The definition of $\alpha_\Omega$-operators and the fact that our operators are of this kind suggest that the vector bundles one can insert in the definition have to satisfy as well some kind of invariance with respect to the twisted action. As is already known, this will imply that the twisting $\alpha$ is necessarily torsion. Let us first give the definition of a twisted vector bundle following \cite{TXL} definition 5.2 and lemma 5.3.

\begin{definition}\label{twistedvbundles}
An $\alpha_\Omega$-vector bundle (complex) is a collection of (complex) vector bundles 
$$\{E_i\to \Omega_i\}_{i\in I}$$
together with an $R_{\alpha_\Omega}$-action such that the center $\bigsqcup_i\Omega_i\times S^1$ acts on $E:=\bigsqcup_iE_i\to \bigsqcup_i\Omega_i$ by scalar multiplication.
\end{definition}

As mentioned above, proposition 5.5 in \cite{TXL} shows that a necessary condition for the $\alpha_\Omega$-vector bundles to exist is that $\alpha$ has to be torsion. We will then assume $\alpha$ is given by a generalized morphism
$$\alpha:\gr --->PU(N)$$
for the rest of the section. Also, in proposition 5.5 ref.cit., Tu and Xu give a very practical equivalence definition which generalizes directly the definition for families given in \cite{MMSI}, the idea is very easy, a $\alpha_\Omega$-vector bundle can be given by a groupoid morphism
$$R_\alpha \to GL_n(\mathbb{C})$$
that descends to a morphism 
$$\gr_\Omega\to P(GL_n(\mathbb{C})),$$
or in other terms $R_\alpha$ is also the pullback extension with respect to the last morphism and to the universal extension 
$$S^1\to GL_n(\mathbb{C})\to P(GL_n(\mathbb{C}))$$
We can now give the definition of projective operators acting on sections of vector bundles.

\begin{definition}
Let $\alpha:\gr --->PU(N)$ be a torsion twisting. Let $E\to \bigsqcup_i\Omega_i$ and 
$F\to \bigsqcup_i\Omega_i$ be two $\alpha_\Omega$-vector bundles. A twisted pseudodifferential operator of order $m$ acting between (twisted) sections of $E$ and $F$ is an operator
 
$$P:C_c^\infty(\gr_\Omega,\Omega^{\frac{1}{2}}\otimes L\otimes t^*E)\to C^\infty(\gr_\Omega,\Omega^{\frac{1}{2}}\otimes L\otimes t^*F).$$
given by a smooth family of order $m$ pseudodifferential operators
$$P_{ij,x}^k:C_c^\infty((\gr_j^k)^x,\Omega^{\frac{1}{2}}\otimes L_{jk}\otimes t^*E_k)\to C^\infty((\gr_i^k)^x,\Omega^{\frac{1}{2}}\otimes L_{ik}\otimes t^*F_k)$$
and such that the following diagram
\begin{equation}
\xymatrix{
C_c^\infty((\gr_j^k)^x,\Omega^{\frac{1}{2}}\otimes L_{jk}\otimes t^*E_k)\ar[d]_-{U_j^r}
\ar[rr]^-{P_{ij,x}^k}&&C_c^\infty((\gr_i^k)^x,\Omega^{\frac{1}{2}}\otimes L_{ik}\otimes t^*F_k)\ar[d]^-{U_i^r}\\
C_c^\infty((\gr_j^l)^y,\Omega^{\frac{1}{2}}\otimes L_{jl}\otimes t^*E_l)\ar[rr]_-{P_{ij,y}^l}&&C_c^\infty((\gr_i^l)^y,\Omega^{\frac{1}{2}}\otimes L_{il}\otimes t^*F_l)
}
\end{equation}
is commutative for every $r\in R_\alpha$ of the form $r=((l,g,k),u)$ with $s(g)=x$, $t(g)=y$.

\end{definition}

\subsection{The algebra of twisted symbols}\label{twistsymbsec} %and quantization}

In this section we will consider a twisting $\alpha^0$ on $M$, it may or not come from a twisting on the groupoid, the results hold anyway. Precisely for this reason, we will denote by $A\to M$ a Lie algebroid over $M$ without any reference to a groupoid.

Consider the space
\begin{equation}
Symb_{cl}^m(A^*,\alpha^0_\Omega):= \bigoplus_{(i,j)}Symb_{cl}^m(N_{ij}^*,L_{ij})
\end{equation}
where $N_{ij}:=A|_{\Omega_{ij}}$.

%where we are denoting as well by $L_{ij}$ the vector bundles over $\Omega_{ij}$ given by restriction of the bundles $L_{ij}$ over $\gr_i^j$. 
By definition one has 
$$Symb_{cl}^m(A^*,\alpha^0_\Omega)\subset Symb_{cl}^n(A^*,\alpha^0_\Omega)$$
for $m\leq n$.
We denote as usual
\begin{center}
$Symb_{cl}^\infty(A^*,\alpha^0_\Omega)=\bigcup_m\,Symb_{cl}^m(A^*,\alpha^0_\Omega)$ and
\end{center}
\begin{center} 
$Symb_{cl}^{-\infty}(A^*,\alpha^0_\Omega)=\bigcap_m\, Symb_{cl}^m(A^*,\alpha^0_\Omega)$
\end{center}
Next, we will define a product on $Symb_{cl}^\infty(A^*,\alpha^0_\Omega)$ using the Fell bundle $L$ over $M_\Omega:=\bigsqcup_{(i,j)}\Omega_{ij}$. Let $a\in Symb_{cl}^m(A^*,\alpha^0_\Omega)$ and $b\in Symb_{cl}^n(A^*,\alpha^0_\Omega)$, we let
$$(a*b)_{ik}\in Symb_{cl}^m(N_{ik}^*,L_{ik})$$ to be defined by
$$(a*b)_{ik}((k,x,i),\xi)=\sum_ja_{ij}((j,x,i),\xi)\cdot_F b_{jk}((k,x,j),\xi)$$
for $x\in \Omega_{ijk}$ and $\xi\in N^*_{ijk}$, and zero otherwise, where $\cdot_F$ stands for the Fell product $L_{ij}^x\otimes L_{jk}^x\stackrel{\cong}{\rightarrow}L_{ik}^x$ and the sum is finite since we are only considering locally finite open coverings $\Omega$. Notice that there is no convolution on the normal direction. The following proposition is an easy exercise proved exaclty as in the classic untwisted case.
\begin{proposition}\label{symbalgebra}
The convolution product described above gives a well defined product on $Symb_{cl}^\infty(A^*,\alpha^0_\Omega)$, with it, this space becomes a filtered algebra. In particular $Symb_{cl}^{-\infty}(A^*,\alpha^0_\Omega)$ is a two sided ideal.
\end{proposition}

As we remarked above the convolution product of the twisted symbol's algebra is only in the space direction not in the normal direction, so in particular if the twisting is trivial this product is only the pointwise product. As in the classic case, the inverse Fourier transform associates to a symbol $a\in Symb_{cl}^*(A^*,\alpha^0_\Omega)$  a distribution \v{a} which as a convolution operator will be a $A$-twisted pseudodifferential operator where we consider $A\rightrightarrows M$ as a Lie groupoid using its vector bundle structure. In fact, we already defined like that our twisted operators with the small difference of the cut off functions. Let us explain all this in detail and enounce the main conclusion. Let $a\in Symb_{cl}^*(N_{ij}^*,L_{ij})$, we let $\text{\v{a}}$ be the distribution acting on 
$u\in C_c^\infty(A_i^j,\Omega^1(A_i^j)\otimes (L_{ij}^*))$ by 
\begin{equation}
\langle \text{\v{a}}, u \rangle:=\int_{(x,X)\in A_i^j}\int_{A_x\times A_x^*}e^{i(X-Y)\cdot \xi}u(x,Y)a(x,\xi)
\end{equation}
Choose a smooth function $\chi(x,X)$ on $A_i^j$ equal to $1$ for $X=0$ and equal to zero for $\|X\|\geq1$, then 
$$\text{\v{a}}=\chi\cdot \text{\v{a}}+(1-\chi)\text{\v{a}}\in \Psi_c^m(A,\alpha_\Omega^0)+\mathscr{S}(A,\alpha^0_\Omega)=:\Psi^*(A,\alpha^0_\Omega),$$
where $\mathscr{S}(A,\alpha^0_\Omega)$ is the twisted groupoid algebra of Schwartz sections of $L\to A_\Omega$, it is constructed exactly as $C_c^\infty (A,\alpha^0_\Omega)$ but admitting sections which are rapidly decreasing in the direction of the fibers of $A$.

On the other hand, given $K\in \Psi^*(A,\alpha^0_\Omega)$ we have the fiberwise Fourier transform that gives a symbol $\hat{K}\in Symb_{cl}^*(A^*,\alpha_\Omega^0).$

\begin{proposition}\label{Fourierpdosymb}
The fiberwise Fourier transform induces an isomorphism of algebras
\begin{equation}
\xymatrix{
\Psi^\infty(A,\alpha^0_\Omega)\ar[r]_-{\cong}^-{\mathscr{F}}&Symb_{cl}^\infty(A^*,\alpha_\Omega^0).
}
\end{equation}
\end{proposition}

\begin{proof}
Using the classic Fourier transform results, we have an isomorphism of vector spaces
\begin{equation}
\xymatrix{
\Psi^*(A_i^j;L_{ij})\ar[r]_-{\cong}^-{\mathscr{F}}&Symb_{cl}^*(N_{ij}^*;L_{ij}).
}
\end{equation}
for every $(i,j)\in I^2$. It extends to an isomorphism of vector spaces 
\begin{equation}
\xymatrix{
\Psi^\infty(A,\alpha^0_\Omega)\ar[r]_-{\cong}^-{\mathscr{F}}&Symb_{cl}^\infty(A^*,\alpha_\Omega^0).
}
\end{equation}
Now, the algebra structure introduced for $Symb_{cl}^\infty(A^*,\alpha_\Omega^0)$ in \ref{symbalgebra} above was precisely computed such that we have an algebra isomorphism.
\end{proof}

{\bf Principal Symbols for projective operators}. We will now see how the principal symbol map extends to our setting. For the rest of the section we come back to the case of a twisting $\alpha$ on a groupoid $\gr$ and the induced twisting $\alpha^0$ on $M$.

For defining the principal symbol is very easy, we just have to recall the Androulidakis-Skandalis general setting about generalized functions with pdo singularities. 
Let $P\in P^m(M,V;E)$. If $P$ is associated with a symbol $a$ of order $m$, then the principal symbol $\sigma_m(P)$ of $P$ is the homogeneous part of $a$ of order $m$, {i.e.}, the class of $a$ in 
$Symb_{cl}^m(N^*;\Omega^1N^*\otimes E)/Symb_{cl}^{m-1}(N^*;\Omega^1N^*\otimes E)$.

We have the following proposition
\begin{proposition}\label{ASpdoext}
The principal symbol map gives a short exact sequence
\begin{equation}
\small{
\xymatrix{
0\ar[r]&\Psi_c^{m-1}(\gr,\alpha_\Omega)\ar[r]& \Psi_c^{m}(\gr,\alpha_\Omega)\ar[r]^-{\sigma_m}&
Symb_{cl}^m(A^*\gr,\alpha_\Omega)/Symb_{cl}^{m-1}(A^*\gr,\alpha_\Omega)\ar[r]&0.
}}
\end{equation}
In particular we have induced isomorphisms
\begin{equation}
\Psi_c^{m}(\gr,\alpha_\Omega)/\Psi_c^{m-1}(\gr,\alpha_\Omega)\cong Symb_{cl}^m(A^*\gr,\alpha_\Omega)/Symb_{cl}^{m-1}(A^*\gr,\alpha_\Omega)
\end{equation}
\end{proposition}
\begin{proof}
It follows from proposition 1.4 in \cite{ASpdo} which states that there is a short exact sequence
\begin{equation}\nonumber
\small{
\xymatrix{
0\ar[r]&P^{m-1}(M,V;E)\ar[r]& P^m(M,V;E)\ar[r]^-{\sigma_m}&Symb_{cl}^m(N^*;\Omega^1N^*\otimes E)/Symb_{cl}^{m-1}(N^*;\Omega^1N^*\otimes E)\ar[r]&0
}}
\end{equation}
\end{proof}

%{\bf Quantization via the tangent groupoid.}

\section{Ellipticity and Regularity properties}

We recall that we are assuming our base manifold $M$ to be smooth and compact.
By picking up a metric on $A\gr$ we can identify as usual the spaces

$$Symb_{cl}^m(N_{ij}^*;L_{ij})/Symb_{cl}^{m-1}(N_{ij}^*;L_{ij})\cong C^\infty_c(S^*N_{ij},L_{ij}).$$

From proposition \ref{ASpdoext} we have a short exact sequence
 \begin{equation}\label{ASpdoext2}
\small{
\xymatrix{
0\ar[r]&\Psi_c^{m-1}(\gr,\alpha_\Omega)\ar[r]& \Psi_c^{m}(\gr,\alpha_\Omega)\ar[r]^-{\sigma_m}&
C_c^\infty (S^*\gr,\pi^*(\alpha_0)_\Omega)\ar[r]&0.
}}
\end{equation}
where $\pi:S^*\gr \to M$ is the canonical projection and $C_c^\infty (S^*\gr,\pi^*(\alpha_0)_\Omega)$ corresponds to the twisted groupoid algebra $\bigoplus_{(i,j)}C^\infty_c(S^*N_{ij},L_{ij})$.

\begin{definition}
The twisted principal symbol of an element $K\in \Psi_c^m(\gr,\alpha_\Omega)$ is the element $\sigma_m(K)\in C_c^\infty (S^*\gr,\pi^*(\alpha_0)_\Omega)$ defined in every component $(i,j)$ as (\ref{ASpdoext2}) above. An operator $K\in \Psi_c^m(\gr,\alpha_\Omega)$ is said to be elliptic if its principal symbol $\sigma_m(K)$ is invertible.
\end{definition}

We collect in the following theorem some classic facts about pseudodifferential operators that generalize to our setting. 

The next theorem is proven following the same proofs as theorem 3.15 and theorem 4.2 in \cite{ASpdo}. 
\begin{theorem}
The following properties hold:
\begin{itemize}
\item[I.]{\bf Symbolic Calculus.} Let $K_a\in \Psi_c^{m_a}(\gr,\alpha_\Omega)$, $a=1,2$, then 
$$\sigma_{m_1+m_2}(K_1* K_2)=\sigma_{m_1}(K_1)* \sigma_{m_2}(K_2).$$
\item[II.]{\bf Parametrix.} Let $K\in \Psi_c^m(\gr,\alpha_\Omega)$ elliptic. There is a pseudodifferential operator $Q\in \Psi_c^{-m}(\gr,\alpha_\Omega)$ such that 
$I-K*Q$ and $I-Q*K$ are in $\Psi_c^{-\infty}(\gr,\alpha_\Omega)$.
%\item[{\bf Square roots.}] Let $K\in \Psi_c^{2m}(\gr,\alpha_\Omega)$ be a self-adjoint even order pseudodifferential operator with $\sigma_{2m}(K)>0$, there is a self-adjoint operator $D\in \Psi_c^m(\gr,\alpha_\Omega)$ such that $K-D^2\in \Psi_c^{-\infty}(\gr,\alpha_\Omega)$.
\end{itemize}
\end{theorem}

%\begin{proof} \end{proof}
%The same regularity properties as in the untwisted case hold in the twisted case
Without surprise we also have the following important result that will allow us to use $C^*$-algebraic and $K$-theoretical tools to explore the analytic indices for projective elliptic operators.

\begin{theorem}\label{pdomultipliers}
Again, with the notations above, we have that
\begin{enumerate}
\item For $m<0$, $\Psi^m(\gr,\alpha_\Omega)\subset C^*(\gr,\alpha_\Omega)$ and, 
\item $\Psi^0(\gr,\alpha_\Omega)\subset M(C^*(\gr,\alpha_\Omega)).$
\end{enumerate}
\end{theorem}

\begin{proof}
This result can be proven following theorem 5.3 in \cite{ASpdo} or its Lie groupoid version \cite{MP,Vass06}, we recall briefly this last approach to show that indeed the same arguments work. 
For proving $(1)$, it is enough to take $K\in P_c^m(\gr_i^j,\Omega_{ij};\Omega^{\frac{1}{2}}\otimes L_{ij})$ with $m< 0$. Let $q\in \mathbb{N}$ the dimension of the source fibers of $\gr$, then it is also the dimension of the sources fibers of $\gr_i^j$ (this is not true for general bi-submersion as in \cite{ASpdo}, this is the reason we can use the classic Lie groupoid arguments in our case). We proceed by induction as follows, if $m<-q$ then by definition of the generalized functions with pseudodifferential singularities the integral of any symbol of order $m$ defining $K$ is a well defined continuous compactly supported section on $\gr_i^j$ and hence it defines an element in the $C^*$-algebra $C^*(\gr,\alpha_\Omega)$. If $m<-\frac{q}{2}$, then $K^*K$ has order $<-q$ and by the above argument it extends to an element in the $C^*-$algebra. Now, its norm, satisfies
$$\|(K^*K)f\|=\|Kf\|^2\leq C\|f\|^2$$
and hence $K$ extends to the $C^*$-algebra as well. In the same way $K^*$ gives an element of the $C^*$-algebra. The induction argument is now easy to follow to get to the conclusion.

For proving $(2)$, one uses the pseudodifferential sequence 
(\ref{ASpdoext}) as in the classic case. Let $P\in \Psi_c^0(\gr,\alpha_\Omega)$, first we can assume that $P$ is given by a distribution $K\in P_c^m(\gr_i^j,\Omega_{ij};\Omega^{\frac{1}{2}}\otimes L_{ij})$ and that $\|\sigma_0(P)\|<1$ in the fibers of $L_{ij}$. We can then choose a projective operator $Q$ such that 
$P^*P+Q^*Q=1+R$ with $R$ of strictly negative order. Then one obtains the following estimation
$$\|Pf\|^2\leq (1+\|R\|)\|f\|^2$$
and so $P$ extends continuously to a multiplier in $C^*(\gr,\alpha_\Omega)$.
\end{proof}

\section{Analytic Index morphism}
We will use the above results to obtain a zero order pseudodifferential extension from which the analytic index will be defined. 

As in the classic case, theorem \ref{pdomultipliers} together with the sequence (\ref{ASpdoext2}) imply that we have the following pseudodofferential extension.

\begin{proposition}
We have the following short exact sequence
\begin{equation}\label{pdoses}
\xymatrix{
0\ar[r]&C^*(\gr,\alpha_\Omega)\ar[r]&\overline{\Psi^0}(\gr,\alpha_\Omega)\ar[r]^-{\sigma}&C^*(S^*\gr,\pi^*(\alpha^0_\Omega))\ar[r]&0
}
\end{equation}
where $\overline{\Psi^0}(\gr,\alpha_\Omega)$ stands for the completion of the zero order operators in $M(C^*(\gr,\alpha_\Omega))$. 
%and $C^*(S^*\gr,\alpha_0+\alpha^*_0)$ 
%for the completion in the Corona algebra $M(C^*(G,\alpha))/C^*(G,\alpha)$ of the sum $\bigoplus_{(i,j)}C^\infty(S^*N_i^j,End(L_{ij}))$.
\end{proposition}

We will take as usual the connecting morphism in $K$-theory of the above short exact sequence of $C^*$-algebras and we will prove that it only depends on the class of $\alpha\in H^1(\gr,PU(H))$ giving thus a sense to the analytic index morphism for the twisted groupoid $(\gr,\alpha)$. We state the precise result:

\begin{proposition}\label{AIndexinvariance}
Let $\alpha_{\Omega_1},\alpha_{\Omega_2}$ be two $PU(H)$-valued $\gr-$cocycles. Suppose they define the same class in $H^1(\gr,PU(H))$ then we have a commutative diagram 
\begin{equation}\label{IndexOmega}
\xymatrix{
K^1(S^*\gr,(\pi^*\alpha^0)_{\Omega_1})\ar[dd]_-{\eta_*}^-\cong\ar[rr]^-{Ind^{a}_{(\gr,\alpha_{\Omega_1})}}&&K^0(\gr,\alpha_{\Omega_1})\ar[dd]^-{\eta_*}_-\cong\\ &&\\
K^1(S^*\gr,(\pi^*\alpha^0)_{\Omega_2})\ar[rr]_-{Ind^{a}_{(\gr,\alpha_{\Omega_2})}}&&K^0(\gr,\alpha_{\Omega_2})
}
\end{equation}
where the isomorphisms $\eta_*$ above are associated in a natural way by an explicit isomorphism $\eta$ between the cocycles, and $Ind^{a}_{(\gr,\alpha_{\Omega_i})}$ stand for the associated connecting morphism of the short exact sequence \ref{pdoses}.
\end{proposition}

\begin{proof}
Giving an explicit isomorphism $\eta$ between the cocycles is equivalent to give a common refinement $\Omega$ of $\Omega_1$ and $\Omega_2$ together with a common cocycle extension, {\it i.e.}, a cocycle $\gr_\Omega\stackrel{\alpha}{\longrightarrow} PU(H)$ with $\alpha|_{\Omega_i}=\alpha_i$, $i=1,2$. By Lemma 3.4 in \cite{CaWangAENS} we know that $\eta$ induces isomorphisms in $K-$theory as required and given by explicit Morita equivalences. We have now to explain why we can choose these isomorphisms $\eta$ to obtain the commutativity of diagram (\ref{IndexOmega}) above. In fact we can choose a sufficiently small refinement so that the above mentioned Morita equivalence is given by a canonical $C^*$-morphism
$$C^*(\gr,\alpha_{\Omega_i})\to C^*(\gr,\alpha_\Omega)$$ where of course the Morita inverse is an honest correspondence (same applies for the groupoid $S^*\gr$). In this situation we also have a $C^*$-morphism (extending by zero)
\begin{equation}\label{PsiOmega}
\overline{\Psi^0}(\gr,\alpha_{\Omega_i})\to \overline{\Psi^0}(\gr,\alpha_\Omega)
\end{equation}
and hence we obtain a morphism between the respective six term exact sequences in which two out of three morphisms, those corresponding to $C^*(S^*\gr,\Omega_i)$ and $C^*(\gr, \Omega_i)$, are isomorphisms. One concludes with a five lemma argument.
\end{proof}
\begin{remark}
In particular the argument of the proof above gives that the K-theory groups \\ $K_*(\overline{\Psi^0}(\gr,\alpha_{\Omega_1}))$ and $K_*(\overline{\Psi^0}(\gr,\alpha_{\Omega_2}))$ are isomorphic. We are not exploring for the moment the Morita invariance or not of these algebras. Notice that the situation above is very particular, we are only moving the cocycle and not the groupoid, indeed if we allow a more general Morita equivalence the analytic index, twisted or not, is not necessarily invariant since the algebroid of a groupoid is not a Morita invariant\footnote{For example a point and a pair groupoid are Morita equivalent groupoids with very different algebroids and very different index morphisms.}.
\end{remark}

%\begin{definition}
%An element  $P\in \overline{\Psi^0}(\gr,\alpha)$ is said to be elliptic if $\sigma(P)$ is invertible. We have in this case a principal symbol class $[\sigma(P)]\in K^1(S^*\gr,\alpha)$ where we denote by $K^1(S^*\gr,\alpha):=K_1(C^*(S^*\gr,\pi^*(\alpha_0+\alpha^*_0)))$.
%\end{definition}

\begin{definition}\label{defanaindex}
Let $\gr\rightrightarrows M$ be a Lie groupoid and $\alpha$ a twisting, the analytic index morphism of $(\gr,\alpha)$ is the connecting morphism of the short exact sequence (\ref{pdoses}) above
\begin{equation}\label{aindex}
\xymatrix{
K^1(S^*\gr,\pi^*\alpha^0)\ar[rr]^-{Ind^{a}_{(\gr,\alpha)}}&&K^0(\gr,\alpha)
}
\end{equation} 
where $K^0(\gr,\alpha):=K_0(C^*(\gr,\alpha))$.
\end{definition}

Thanks to above proposition the index morphism above only depends on the twisting $\alpha$ on $\gr$. In practice of course one often choose a representing cocycle. From now on we will then drop from our notations the reference to the open covers and hence to the Cech groupoids.

Now, we will see that the index morphism above can be factorize, as in the classic case, via the twisted $K^0$-group of the cotangent Lie algebroid, giving thus a more primitive and tractable index.

By considering the induced twisting on the Lie algebroid and the associated twisted groupoid $(A\gr,\pi^*\alpha^0) $ we have the pseudodifferential extension
\begin{equation}\label{pdoses}
\xymatrix{
0\ar[r]&C^*(A\gr,\pi^*\alpha^0)\ar[r]&\overline{\Psi^0}(A\gr,\pi^*\alpha^0)\ar[r]^-{\sigma}&C^*(S^*\gr,\pi^*\alpha^0)\ar[r]&0
}
\end{equation}
and the corresponding index morphism
\begin{equation}\label{aindexalgebroid}
\xymatrix{
K^1(S^*\gr,\pi^*\alpha^0)\ar[rr]^-{{\delta}}&&K_{top}^0(A^*\gr,\pi^*\alpha^0)
}
\end{equation} 
where $K_{top}^0(A^*\gr,\pi^*\alpha^0)$ is the topological Twisted K-theory of $(A^*\gr,\pi^*\alpha^0)$. The above follows since $S^*(A^*\gr)=S^*\gr$ and since $K^0(A\gr,\pi^*\alpha^0)\cong K_{top}^0(A^*\gr,\pi^*\alpha^0)$ via a Fourier isomorphism of the corresponding $C^*-$algebras (proposition 2.11 in \cite{CaWangAdv}).

%\begin{theorem}
%Given $P\in \overline{\Psi^0}(\gr,\alpha)$ elliptic, the principal symbol defines a class 
%$[\delta(\sigma(P))]\in K_{top}^0(A^*\gr,\pi^*\alpha)$. Moreover, every class in $K_{top}^0(A^*\gr,\pi^*\alpha)$ is obtained in this way.
%In particular $\delta$ is surjective.
%\end{theorem}

Now, remember that in  \cite{CaWangAdv} we used the Connes tangent groupoid approach to construct an index morphism
\begin{equation}\label{defind}
\xymatrix{
K^0(A^*\gr,\pi^*\alpha^0)\ar[rr]^-{Ind_{(\gr,\alpha)}}&&K^0(\gr,\alpha)
}
\end{equation} 
which generalizes the analytic index morphism for a (untwisted) Lie groupoid. We briefly recall its construction: The main, very simple observation, is that the functoriality of the deformation to the normal cone construction implies that the twisting $\alpha$ on $\gr$ extends to a twisting $\alpha^T$ on $\gr^T$ such that $\alpha^T|_{t\neq 0}$ identifies with $\alpha$ while $\alpha^T|_{t=0}$ identifies with the twisting $\pi^*\alpha_0$ on $A\gr$ coming from the twisting $\alpha_0$ on $M$. There is then a short exact sequence of $C^*$-algebras
\begin{equation}\label{pdoses}
\xymatrix{
0\ar[r]&C^*(\gr\times (0,1],\alpha^T|_{(0,1]})\ar[r]&C^*(\gr^T,\alpha^T)\ar[r]^-{ev_0}&C^*(A\gr,\pi^*\alpha^0)\ar[r]&0
}
\end{equation}
with contractible kernel. The analytic deformation index is defined as 
\begin{equation}
\xymatrix{
K^0(A^*\gr,\pi^*\alpha^0)\ar[rrr]^-{Ind_{(\gr,\alpha)}:=e_1\circ e_0^{-1}}&&&K^0(\gr,\alpha)
}
\end{equation}
where $e_t$ stands for the morphism in $K$-theory induced from the evaluation at $t$, see also \cite{Ca13} more complementary details.

We have the following theorem which stays that the analytic index morphism using pseudodifferential calculus factors through the analytic index (or deformation index) using Connes tangent groupoid. 

\begin{theorem}\label{DefindexvsAnaindex}
Given $(\gr,\alpha)$ as above, we have the following commutative diagram
\begin{equation}
\xymatrix{
K^1(S^*\gr,\pi^*\alpha^0)\ar[d]_-{\delta}\ar[rr]^-{Ind^{a}_{(\gr,\alpha)}}&&K^0(\gr,\alpha)\\
K^0(A^*\gr,\pi^*\alpha^0)\ar[rru]_-{Ind_{(\gr,\alpha)}}&&
}
\end{equation}
\end{theorem}
\begin{proof}
We have the following commutative diagram
\[
\xymatrix{
0\ar[r]&C^*(\gr,\alpha)\ar[r]&\overline{\Psi^0}(\gr,\alpha)\ar[r]^-{\sigma}&C^*(S^*\gr,\pi^*\alpha^0)\ar[r]&0\\
0\ar[r]&C^*(\gr^T,\alpha^T)\ar[r]\ar[u]^-{e_1}\ar[d]_-{e_0}&\overline{\Psi^0}(\gr^T,\alpha^T)\ar[r]^-{\sigma}\ar[u]^-{e_1}\ar[d]_-{e_0}&C^*(S^*\gr^T,\alpha^T)\ar[r]\ar[u]^-{e_1}\ar[d]_-{e_0}&0
\\
0\ar[r]&C^*(A\gr,\pi^*\alpha^0)\ar[r]&\overline{\Psi^0}(A\gr,\pi^*\alpha^0)\ar[r]^-{\sigma}&C^*(S^*\gr,\pi^*\alpha^0)\ar[r]&0
}
\]
In $K-$theory the evaluations at zero induce isomorphisms, indeed 
$C^*(\gr^T,\alpha)\stackrel{e_0}{\to}C^*(A\gr,\pi^*\alpha^0)$ induces an isomorphism as recalled in the construction of the deformation index, 
$C^*(S^*\gr^T,\alpha^T)\stackrel{e_0}{\to}C^*(S^*\gr,\pi^*\alpha^0)$ induces an isomorphism since $S^*\gr^T=S^*\gr\times [0,1]$ and since, over $S^*\gr^T$, $\alpha^T=\pi^*\alpha^0\times id_{[0,1]}$ (with obvious meaning), and finally $\overline{\Psi^0}(\gr^T,\alpha^T)\stackrel{e_0}{\to}\overline{\Psi^0}(A\gr,\pi^*\alpha^0)$ induces an isomorphism by the five lemma. The conclusion follows immediately.
\end{proof}

\subsection{The Twisted longitudinal Connes-Skandalis index theorem.}

Let $(M,F)$ be a foliated compact manifold. Let $\alpha$ be a twisting on the holonomy groupoid $\gr$. In \cite{CaWangAdv} the Connes-Skandalis twisted topological index morphism was constructed
$$Ind^{top}_{(M,F),\alpha}:K^0(F^*,\pi^*\alpha^0)\to K^0(\gr,\alpha),$$
as an immediate generalization of Connes-Skandalis topological index morphism using as them an embedding of $M$ into an euclidean space but adapting the Thom isomorphism to the twisted case.  Our main theorem in
\cite{CaWangAdv} (theorem 3.3) is the equality between the analytic index morphism constructed through the tangent groupoid and the topological index morphism. The next result is a consequence of theorem 3.3 in ref. cit. and theorem \ref{DefindexvsAnaindex} above: 

\begin{corollary}[Twisted Connes-Skandalis for projective families of longitudinal operators]\label{twistCSthm}
Let $(M,F)$ be a foliated compact manifold. Let $\alpha$ be a twisting on the holonomy groupoid $\gr$ (without any restriction on the twisting). Let 
$D\in \Psi^*(\gr,\alpha_\Omega)$ a projective elliptic longitudinal pseudodifferential operator, then

Then we have the following equality of $K-$theory morphisms
\begin{equation}
Ind^a_{(M,F),\alpha}(\sigma(D))=Ind^{top}_{(M,F),\alpha}(\delta(\sigma(D))),
\end{equation}
where $\sigma(D)\in K^1(S^*F,\alpha)$ is the class of the principal symbol class and $\delta(\sigma(D))$ its image on $K^0(F^*,\alpha)$.
\end{corollary}

\section{Examples of projective pseudodifferential operators}

\subsection{Projective families of Dirac operators}
The following example largely generalizes the projective families of Dirac operators introduced in \cite{MMSI} and on \cite{BG}. We will consider below two explicit subexamples covered by this situation, one of which includes the case of families treated in ref. cit.

Let $\gr$ be a Lie groupoid and $\alpha$ a torsion twisting. Let $E$ be a Hermitian $\mathbb{Z}_2$-graded $\alpha$-vector bundle. As shown in \cite{MMSI} or 
\cite{BG} p.10, the collection $End(E_i)$, the $E_i$ as in definition \ref{twistedvbundles},  defines a bundle of algebras over $M$ which we can denote of course $End(E)$.

We say that $E$ is a twisted Clifford bundle if there is an homomorphism 
$$c:Cliff(A\gr)\to End(E)$$
of unital $\mathbb{Z}_2$-graded *-algebras.

Suppose we have a Clifford connection $\nabla^E$ on $E$, that is, a collection of hermitian $A\gr$-connections $\nabla^i$ such that 
$$\nabla^i_X(cl(v)\cdot e_a)=cl(\nabla_Xv)e_a+cl(v)(\nabla^i_Xe_a).$$

We can construct as classically an operator
$$D_j:=cl\circ \nabla^i:\Gamma(E_j)\longrightarrow\Gamma(E_j)$$
for every $j$ by using the Clifford module structure. 

Now, over $\gr_j^k$ we have an isomorphism $L_{jk}\otimes t^*E_k\cong s^*E_j$ given by the $R_\alpha$-action on $E$. Once restricted to $(\gr_j^k)^x$ we can consider the operators
$$\mathscr{D}_{j,x}^k:\Gamma(L_{jk}\otimes t^*E_k)\longrightarrow\Gamma(L_{jk}\otimes t^*E_k)$$
given by pullback by $s:\gr_j^k\to \Omega_j$ all the structures above used to construct $D_j$ and using the isomorphisms $L_{jk}\otimes t^*E_k\cong s^*E_j$ and $s^*A^*\gr\cong T^*_t\gr$ (where $T_t\gr$ stands for the vertical tangent bundle with respect to the submersion 
$t:\gr\to M$). In order for the family $\mathscr{D}_{j,x}^k$ to define an order one twisted differential operator we require a little bit more on the connections. Indeed the connections $\nabla^i$ should be compatible in some way with the twisting. In fact we require that under the isomorphism $L_{ik}\otimes t^*E_k\cong s^*E_i$ one has
$$s^*\nabla^i=id\otimes t^*\nabla^i +\nabla^{ik}\otimes id$$
where $\nabla^{ik}$ are connections on $L_{ik}$ satisfying the "Fell condition"
$$m^*\nabla^{ik}=\nabla^{ij}\otimes id +id \otimes \nabla^{jk}$$
under the isomorphism of bundles $L_{ij}\otimes L_{jk}\cong m^*L_{ik}$ over $\gr_i^j\,_{t_j}\times_{s_j}\gr_j^k$ (whenever not empty) and with $m:\gr_i^j\,_{t_j}\times_{s_j}\gr_j^k\to \gr_i^k$ the groupoid multiplication. Connections as above exist, the proof follows the same lines as lemma 2.2 and lemma 2.11 in \cite{BG}.

We have finally: Given a twisted Clifford bundle as above we can form a familiy of Dirac operators $\mathscr{D}_{j,x}^k$ that gives an operator
$$\mathscr{D}_E\in \Psi^1((\gr,\alpha);E)$$
that can be called "The twisted Dirac operator for the twisted Clifford bundle $E$". As in the untwisted case, the index theoretical interesting part of the operator above is its positive part
$\mathscr{D}_E^+\in \Psi^1((\gr,\alpha);E^+,E^-)$ whose analytic index lives in 
$$Ind^a_{\gr,\alpha}(\mathscr{D}_E^+)\in K_0(C^*_r(\gr,\alpha_\Omega)).$$

\subsubsection{An example for Riemannian foliations}. Let $(M,F)$ be a Riemannian foliation with holonomy groupoid $\gr$, suppose that the normal bundle $N$ has a $\gr$-invariant metric and consider the orientation twisting it defines
$$\gr\stackrel{\alpha_N}{--->}PU(H).$$
Suppose $N$ has even rank. An example of a $\mathbb{Z}_2$-graded $\alpha_N$-vector bundle can be constructed from the local spinors of $N$, let us denoted them by $S^+$ and $S^-$. The action of $\gr$ on $N$ gives the Clifford action
$$Cl(F)\to End(S^+,S^-)$$
The Dirac operator construction above yields an operator
$$D_S\in \Psi^1((\gr,\alpha_N);S)$$
and an analytic index
$$Ind^a_{\gr,\alpha}(\mathscr{D}_S^+)\in K_0(C^*_r(\gr,\alpha_\Omega)),$$
that might be computed topologically by means of the twisted longitudinal index theorem above, {i.e.},
$$Ind^a_{\gr,\alpha}(\mathscr{D}_S^+)=Ind^{top}_{\gr,\alpha}([\sigma(\mathscr{D}_S^+)]),$$
where $[\sigma(\mathscr{D}_S^+)]\in K^0(F^*,\alpha)$ is its principal symbol class.

\subsubsection{Projective representations for discrete groups and twisted operators on coverings}

Let $\Gamma$ be a discrete group and let $\alpha:\Gamma \to PU(H)$ be a projective representation. Let $M$ be a closed smooth manifold and $f:M--->\Gamma$ be a generalized morphism. Consider the associated $\Gamma-$covering $\widetilde{M}\to M$  and the associated Connes-Moscovici groupoid (\cite{Concg} $III.4\alpha$)
$$\widetilde{M}\times_\Gamma \widetilde{M}\rightrightarrows M,$$
whose Lie algebroid is $TM$.

There is an explicit Morita equivalence $\tilde{f}:\widetilde{M}\times_\Gamma \widetilde{M}--->\Gamma$ such that the following diagram of generalized morphisms commutes
\begin{equation}
\xymatrix{
M\ar[d]_-{p}\ar[r]^-f&\Gamma\\
\widetilde{M}\times_\Gamma \widetilde{M}\ar[ru]_-{\tilde{f}}&
}
\end{equation}

We will consider the twisting $\tilde{\alpha}:=(\alpha\circ \tilde{f})$ on the Lie groupoid $\widetilde{M}\times_\Gamma \widetilde{M}$. 

Consider an even rank vector bundle over $M$ defined by a cocycle 
$$M\stackrel{O_E}{--->}SO(n)$$ and suppose it passes to the groupoid $\widetilde{M}\times_\Gamma \widetilde{M}\stackrel{O_E}{--->}SO(n)$, or in other words it is a $\Gamma$-invariant vector bundle over $M$. Next, consider the composition
$$\widetilde{M}\times_\Gamma \widetilde{M}\stackrel{O_E}{--->}SO(n)\stackrel{\beta}{\rightarrow}PU(H)$$
where $\beta$ is induced from $Spin^c(n)\to U(H)$ as explained in example \ref{obundle} above.
Finally, assume this bundle is compatible with the twisting, {\it i.e.}, 
$$\beta\circ O_E=\tilde{\alpha}$$
or $E$ is an $\alpha$-twisted vector bundle, this implies the twisting $\alpha$ has to be torsion. By choosing local liftings to $Spin^c(n)$ we have a twisted action
$$Cl(TM)\to End(E)$$
and a twisted Dirac operator
$$D_\alpha\in \Psi^1(\widetilde{M}\times_\Gamma \widetilde{M},\tilde{\alpha}),$$
whose index lives in 
$$Ind_a(D_\alpha^+)\in K^0(\widetilde{M}\times_\Gamma \widetilde{M},\tilde{\alpha})\sim K_0(C^*(\Gamma,\alpha)).$$
%Under the assumption that $M$ is a spin even dimensional manifold we will construct next the Twisted Dirac operator
%$$D_\alpha\in \Psi^1(\widetilde{M}\times_\Gamma \widetilde{M},\tilde{\alpha}),$$
%whose index lives in 
%$$ind_a(D_\alpha)\in K^0(\widetilde{M}\times_\Gamma \widetilde{M},\tilde{\alpha})\sim K_0(C^*(\Gamma,\alpha)).$$

We beleive that this example is closely related to the twisted operators considered by Azzali and Wahl in \cite{AW} that were first worked out by Mathai in \cite{Ma1,Ma2} to obtain very interesting geometric corollaries. We will try to study this somewhere else.

\subsection{Projective symbols of Fractional Indices' projective pseudodifferential operators}

%We will show how symbols of the projective operators considered in \cite{MMSfrac} can be interpreted as projective operators in our sense.

In \cite{MMSfrac}, Mathai, Melrose and Singer showed that any oriented manifold admits a projective Dirac operator even if the manifold does not admit a spin structure, in this case they show the $\hat{A}-$genus is still computed by the index of this operator and hence it is a rational number. In fact, in the same paper, they prove a topological index formula for every pseudodifferential projective operator acting between the sections of twisted vector bundles associated to a twisting (equivalently, a finite rank Azumaya bundle in their terms) on the manifold. The indices for these operators are hence rational numbers. In principle their analytic index is a map from the twisted $K$-theory of the cotangent bundle to the real numbers.

Now, to make the link with our paper we have to clarify some terminology. The projective pseudodifferential operators treated in the fractional index paper \cite{MMSfrac} are not a particular case of ours, in fact, the projective families operators introduced in \cite{MMSI} by Mathai, Melrose and Singer are not the fractional families version corresponding to the fractional indices operators. We recall that our operators generalizes the operators introduced in \cite{MMSI} for families. 

Let us introduce some notation to better explain the above paragraph and get to the description of our main example. Let $\gr\rightrightarrows M$ be a Lie groupoid. In this paper we start always with a (groupoid) twisting on $\gr$ that induces a twisting on $M$ and on $A^*\gr$. Now, if one starts with a twisting on the base manifold $M$, it induces by pullback a twisting on $A^*\gr$, but the twisting does not necessarily extend to the entire groupoid $\gr$. For example if $\gr=M\times M\rightrightarrows M$ is the pair groupoid, every twisting on $\gr$ is trivial while the twistings on $M$ are classified by $H^3(M;\mathbb{Z})$. For this precise example, what Mathai, Melrose and Singer exploit in \cite{MMSfrac} is that the twisting on 
$M$ extends in some way to a sufficiently small neighborhood of the diagonal and define projective pseudodifferential operators with a restriction on the support which depends on the neighborhood of the diagonal. 

To get to our example, let $\alpha^0$ be a torsion twisting on a manifold $M$. Denote by $\Psi^\infty_\epsilon(M;E)$ the algebra of "Fractional Indices" projective pseudodifferential operators constructed in \cite{MMSfrac}. By construction, or see \cite{SS} (p. 312 equation (3.1)) for a nice explicit computation, this algebra of projective operators has as associated algebra of symbols what we have denoted as
$$Symb_{cl}^\infty(T^*M;\alpha^0).$$
By our proposition \ref{Fourierpdosymb}, we have a Fourier isomorphism
\begin{equation}
\xymatrix{
\Psi^\infty(TM,\alpha^0)\ar[r]_-{\cong}^-{\mathscr{F}}&Symb_{cl}^\infty(T^*M,\alpha^0),
}
\end{equation}
or in other words the algebra of symbols for the "Fractional Indices" projective pseudodifferential operators can be obtained as an algebra of projective total symbols in the sense of the present paper.

But notice further, that given $\alpha^0$ a twisting (torsion or not) on a manifold $M$ and $A\to M$ a Lie algebroid, our proposition \ref{Fourierpdosymb} still holds. One can then ask for a pseudodifferential quantization for the associated algebra of symbols and the associated Index theory. In the case the algebroid $A=A\gr$ is integrable by a groupoid $\gr$ and the twisting extends to $\gr$ it corresponds to the theory we developed in this article. The other very interesting case is when the twisting does not extend to $\gr$, in this case one expects to obtain the higher fractional index theory. We will discuss and develop this topic elsewhere.

\bibliographystyle{amsplain}
\bibliography{bibliographie}

\end{document}